%% file: paper.tex
\documentclass[11pt,a4paper]{amsart}
\usepackage[dvips]{epsfig}
\usepackage{graphics}
\usepackage{latexsym}
\usepackage{verbatim}
\usepackage{amsmath}
\usepackage{amsthm}
\usepackage{amssymb}
\usepackage[]{hyperref}
\hypersetup{
    colorlinks=true,%
}
\usepackage [table]{xcolor}
\usepackage{multirow}
\usepackage{float}
\usepackage{tikz}
\usepackage{subfig}
\newtheorem{theorem}{Theorem}[section]

\newtheorem{proposition}[theorem]{Proposition}
\newtheorem{definition}[theorem]{Definition}

\newcommand{\RR}{\mathbb{R}}
\newcommand{\NN}{\mathbb{N}}
\newcommand{\F}{\mathcal{F}}
\newcommand{\G}{\mathcal{G}}
\newcommand{\E}{\mathcal{E}}
\newcommand{\C}{\mathcal{C}}
\usepackage{algorithmic}

\def\signff{\bigskip\bigskip\hspace{80mm}
\vbox{{\sc Francis Filbet\par\vspace{3mm}
Universit\'e de Lyon,\par
UL1, INSAL, ECL, CNRS \par 
UMR5208, Institut Camille Jordan,\par
43 boulevard 11 novembre 1918,\par
F-69622 Villeurbanne cedex,  FRANCE
\par\vspace{3mm}e-mail:} filbet@math.univ-lyon1.fr }}

\def\signcn{\bigskip\bigskip\hspace{80mm}
\vbox{{\sc Claudia Negulescu\par\vspace{3mm}
Universit\'e de Provence,\par
39, rue Joliot Curie, \par 
13453 Marseille Cedex, FRANCE
\par\vspace{3mm}e-mail:} claudia.negulescu@cmi.univ-mrs.fr  }}

\def\signcy{\bigskip\bigskip\hspace{80mm}
\vbox{{\sc Chang Yang \par\vspace{3mm}
Laboratoire Paul Painlev\'e U.M.R CNRS 8524, \par 
Universit\'e Lille 1 -- Sciences et Technologies, \par 
Cité Scientifique 59655, \par 
59650 Villeneuve d'Ascq Cedex, FRANCE
\par\vspace{3mm}e-mail:} chang.yang@math.univ-lille1.fr  }}

\begin{document}

\title[Numerical study of a nonlinear heat equation for plasmas]{Numerical study of a nonlinear heat equation for  plasma physics}\thanks{F. Filbet is partially supported by the European Research Council ERC Starting Grant 2009,  project 239983-\textit{NuSiKiMo}, C. Negulescu is partially supported by the ANR project ESPOIR}

\author{Francis Filbet, Claudia Negulescu and Chang Yang}

\hyphenation{bounda-ry rea-so-na-ble be-ha-vior pro-per-ties
cha-rac-te-ris-tic}

\maketitle

\begin{abstract}
This paper is devoted to the numerical approximation of a nonlinear temperature balance equation, which describes the heat evolution of a magnetically confined  plasma in the edge region of a tokamak. The nonlinearity implies some numerical difficulties, in particular long time behavior, when solved with standard methods. An efficient numerical scheme is presented in this paper, based on a combination of a directional splitting scheme and the IMEX scheme introduced in \cite{bibAclass}.
\end{abstract}

\noindent {\sc Keywords.}{ Nonlinear heat equation, IMEX scheme, finite volume method} 



\section{\bf Introduction} 
\setcounter{equation}{0}
\label{sec:1}
The description and simulation of the transport, especially the turbulence of magnetically confined fusion plasmas in the edge region called scrape off layer (SOL) of a tokamak, is nowadays one of the main problems for fusion generated energy production (ITER). The understanding of the physics in this edge region is fundamental for the performances of the tokamak, in particular the plasma-wall interactions as well as the occurring turbulence have an important impact on the confinement properties of the plasma. From a numerical point of view, an accurate approximation of the plasma evolution in the edge region is essential since energy fluxes as well as  particle fluxes at the boundary are used as boundary conditions for the mathematical model applied to describe the plasma evolution in the center region (core) of the tokamak. The physical properties of these two regions (core/edge) are rather different, so that different models are used for the respective plasma-evolution modeling: the gyrokinetic approach for the collisionless core-plasma and the fluid approach for the collisional edge-plasma.\\

A large variety of models can be found in literature \cite{livia,patrick_ph} for the description  of the SOL, based on various assumptions and aimed to describe different physical phenomena. We shall concentrate in this paper on the TOKAM3D model, introduced in \cite{patrick}. The aim of this model is the investigation of the instabilities occurring in this plasma edge region, as for example the Kevin-Helmholtz instability, the electron-temperature-gradient (ETG), ion-temperature-gradient instabilities (ITG) , etc.
\\
The TOKAM3D model is based on a two-fluid description (ions, electrons) and consists of the usual continuity equation, equation of motion and energy balance equation, closed by the so-called ``Braginskii closure''. These equations are
\begin{equation} \label{EQ_MF}
\left\{
\begin{array}{l}
\displaystyle \partial_t n_\alpha + \nabla \cdot (n_\alpha u_\alpha) =S_{n\alpha}\,, \\[3mm]
\displaystyle m_\alpha n_\alpha \left[ \partial_t u_\alpha + (u_\alpha \cdot \nabla ) u_\alpha \right] = - \nabla p_\alpha + n_\alpha e_\alpha (E + u_\alpha \times B) - \nabla \cdot \Pi_\alpha + R_\alpha\,, \\[3mm]
{3 \over 2} n_\alpha \left[ \partial_t T_\alpha + (u_\alpha \cdot \nabla ) T_\alpha \right] + p_\alpha  \nabla \cdot u_\alpha = - \nabla \cdot q_\alpha - \Pi_\alpha : \nabla u_\alpha + Q_\alpha\,,
\end{array}
\right.
\end{equation}
where $n_\alpha$ is the particle density ($\alpha=e$ for electrons and $\alpha=i$ for ions), $u_\alpha$ the velocity, $\Gamma_\alpha:= n_\alpha u_\alpha$ the particle flux, $m_\alpha$ the particle mass, $ e_\alpha$ the particle charge ($e_e=-1$ for electrons and $e_i=1$ for ions), $p_\alpha$ the pressure, $\Pi_\alpha$ the stress (viscosity) tensor, $S_{n \alpha}$ a particle source term (coming from the core plasma), $R_\alpha$ the friction force due to collisions, $T_\alpha$ the temperature, $q_\alpha$ the energy flux and finally $Q_\alpha$ the particle exchange energy term, due to collisions. In the Braginskii closure, the pressure is specified as $p_\alpha:=n_\alpha T_\alpha$ (perfect gas assumption), the plasma viscosity is supposed negligible, such that $\nabla \cdot \Pi_\alpha=0 $ and $\Pi : \nabla u_\alpha=0$ and the energy flux $q_\alpha$ is supposed to have a diffusive form, given in terms of the temperature gradient, as follows $q_\alpha:= - \kappa_\alpha \nabla T_\alpha$ (coming from the Fourier law) with $\kappa_\alpha$ the thermal conductivity coefficient. The energy exchange term $Q_\alpha$ is taken under the form
$$
Q_\alpha:= \pm 3 {m_e \over m_i} {n_\alpha \over \tau_e} (T_e -T_i)\,,
$$
where $\tau_e$ is the electron-ion collision time.\\
Due to the high complexity of the problem, several other hypothesis are assumed, permitting to concentrate on the desired features and to filter out the insignificant/disturbing details. These hypothesis, as for example the quasi-neutrality $n_e \sim n_i$, are not detailed here and we refer the reader to the more physical works \cite{briz,patrick_ph,wesson}. \\

Several difficulties arise when trying to solve numerically the system (\ref{EQ_MF}). We shall concentrate in this paper only on the temperature equation, which requires  at the moment still a lot of effort, due to its inherent numerical burden. The resolution of the two other equations was the aim of the PhD thesis \cite{patrick}. The numerical difficulties in solving the temperature equation are firstly related to the thermal conductivity coefficients, which depend on the temperature itself, leading thus to a non-linear problem. Secondly, the strong magnetic field which confines the tokamak plasma introduces a sharp anisotropy into the problem. Indeed, the charged particles gyrate around the magnetic field lines, moving thus freely along the field lines, but their dynamics in the perpendicular directions is rather restricted. Quantities as for example the resistivity or the conductivity, differ thus in several orders of magnitude when regarded in the parallel or perpendicular directions. Finally, boundary conditions have to be imposed, which is a rather delicate task from a physical, mathematical and numerical point of view.\\

Let us now present in more details the model we are interested in. In this paper, we shall study a simplified version of the temperature evolution equation, which contains however all the numerical difficulties of this last one. We shall focus on how to handle with the nonlinear terms and the boundary conditions, the high anisotropy being the aim of a forthcoming work \cite{BMN,MNN}. The simulation domain $\Omega=(0,1)\times(0,1)$ with boundary $\partial \Omega$ is presented in Figure \ref{ima1}. It consists of a periodic core region, separated by a Separatrix from the non-periodic SOL region. Its axes represent the direction parallel to the magnetic field lines ($s$) and the radial direction ($r$). We assume in this paper that all quantities are invariant with respect to the poloidal angle $\varphi$. The parallel thermal conductivities $\kappa_{s,||}$ depend on  $T_\alpha^{5/2}$ whereas the perpendicular ones $\kappa_{s,\perp}$, governed by the turbulence, are independent of the temperature \cite{bragi}.\\

The system we are interested in, is composed of the evolution equation
\begin{equation} 
\label{temp_simpl}
\partial_t T_\alpha - \partial_s (K_{||,\alpha}\, T^{5/2}_\alpha\, \partial_s T_\alpha) - \partial_r ( K_{\perp,\alpha} \, \partial_r T_\alpha) = \pm \beta_\alpha(T_e - T_i)\,, \quad  \,\, (s,r) \in \Omega\,,
\end{equation}
completed with the boundary conditions
\begin{equation} \label{BC}
\left\{
\begin{array}{l}
 \displaystyle \partial_r T_\alpha =  -{Q_{\perp,\alpha}}\,, \,\, r=0\,, \,s \in (0,1)\,,
\\
\,
\\
\displaystyle\partial_r T_\alpha = 0 \,, \,\, r=1\,, \, s \in (0,1)\,,
\\
\,
\\
\displaystyle K_{||,\alpha}\, T^{5/2}_\alpha\, \partial_s T_\alpha =  \,\gamma_\alpha \, T_\alpha, \,\, r \in (1/2,1)\,, \, s=0\,,
\\
\,
\\
\displaystyle K_{||,\alpha}\, T^{5/2}_\alpha\, \partial_s T_\alpha = -\gamma_\alpha \,  T_\alpha, \,\, r \in (1/2,1)\,, \, s=1 \,,
\\
\,
\\
\displaystyle T(t,0,r)= T(t,1,r)\,, \,\, r \in (0,1/2),
\end{array}
\right.
\end{equation}
and the initial condition
\begin{equation} \label{IC}
T_\alpha(0)=T_{\alpha,0} \ge 0\,.
\end{equation}
The diffusion parameters $0<K_{\perp,\alpha}\ll K_{||,\alpha}$ and the core-heat flux $Q_{\perp,\alpha}>0$ are considered as given. The non-linear boundary conditions at the limiter express the fact, that we have continuity of the heat fluxes at the boundary. Indeed, the heat flux $q:= \gamma \Gamma_{||} T$ at the boundary is given as the sum of a diffusive and a convective term, like
$$
\gamma\, \Gamma_{||}\, T = {5 \over 2}\, \Gamma_{||}\, T - | \kappa_{||} |\, \partial_s T\,, \quad \Gamma_{||} = n \, u_{||}\,.
$$ 
At $s=0$ the particle velocity $u_{||} <0$ is negative, whereas at $s=1$ we have $u_{||} >0$, which gives rise to the boundary conditions in (\ref{BC}). The constant $\gamma_\alpha$ is different for electrons and ions, in particular $\gamma_i \sim 0$ for ions and $\gamma_e \sim 5/2$ for electrons. In the case of ions, we have thus homogeneous Neumann boundary conditions at the limiter.

\vspace{0.2cm}
\begin{figure}[htbp]
\begin{center}
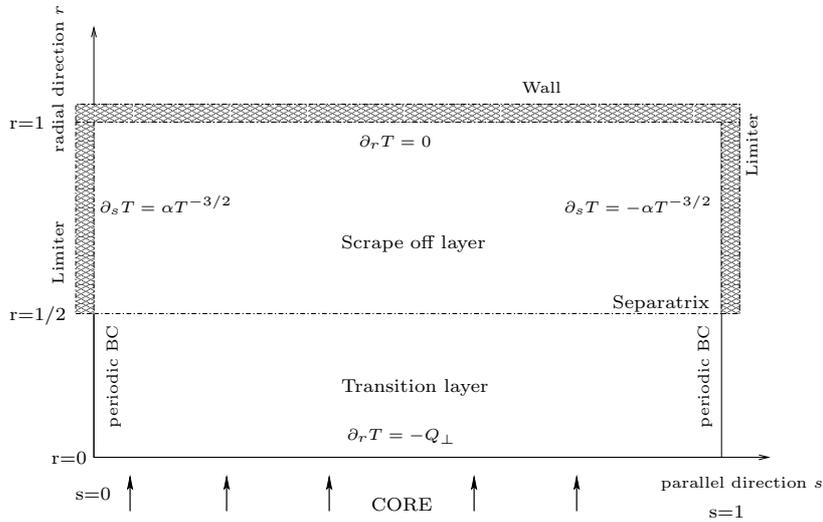
\end{center}
\caption{\label{ima1} The 2D domain.}
\end{figure}
\vspace{0.2cm}

The outline of this paper is the following. In Section \ref{sec:2},  we will focus on the 1D nonlinear parabolic problem
\begin{equation*}
  \partial_tT-\partial_s(K_\parallel T^{5/2}\partial_sT)=0,
\end{equation*}
completed with the nonlinear boundary conditions in $s=0,1$. A mathematical study is firstly performed. Then, explicit, implicit and IMEX schemes are compared for the resolution of this 1D problem, with respect to precision and simulation time. In Section \ref{sec:3} we consider the complete 2D problem for one species (without the source term). A directional Lie splitting method is used in order to transform the 2D problem in two 1D problems and to apply the results of the previous section. Finally, in Section \ref{sec:4} we solve the complete 2D ion-electron coupled problem. The shapes of the different electron/ion temperatures are compared.

\section{\bf The 1D nonlinear problem}
\label{sec:2}
\setcounter{equation}{0}
Let us consider in this section the 1D nonlinear problem,
corresponding to the temperature balance equation in the parallel direction, {\it i.e.}
\begin{equation}
  \left\{
  \begin{array}{ll}
 \displaystyle   \partial_tT-\partial_s(K_\parallel |T|^{5/2}\partial_sT)=0,\,\,(t,s)\in \RR^+\times (0,1), 
\\
\,
\\
  \displaystyle   K_{\parallel}|T|^{5/2}\partial_sT=\gamma T,\, s=0, 
\\
\,
\\
  \displaystyle   K_{\parallel}|T|^{5/2}\partial_sT=-\gamma T,\,s=1,
\\
\,
\\
  \displaystyle   T(0,\cdot)=T^0,
  \end{array}
\right.
\label{eq1D}
\end{equation}
where $\gamma \ge 0$ is a given constant, $K_{||} \in
L^{\infty}(\Omega)$, $T^0 \in L^{2}(\Omega)$, with $T^0 \ge 0$ and  $K_{||}\ge 0$ almost everywhere. Let us denote in this section the domain by $\Omega=(0,1)$ and the time-space cylinder by $Q\,:= \RR^+ \times (0,1)$. The aim of this section is to study from a mathematical point of view this equation and to introduce an efficient numerical scheme for its resolution. From a physical point of view, problem
(\ref{eq1D}) describes
the rapid diffusion process of the initial temperature $T^0$ and the
outflow through the boundary.
\subsection{Mathematical study}
Before starting with the numerical discretization, we first establish some properties of the 1D diffusion problem (\ref{eq1D}), like
existence, uniqueness of a solution, positivity etc. To simplify the presentation, we shall assume for the present study that $K_{||} \equiv 1$, the general case being treated equally.\\

We also denote $p >2$ and $p'$ its conjugate number $1< p':={p \over p-1}<2$. The
diffusion coefficient can now be written as $a(T):= |T|^{p-2}$. Moreover, let us define the primitive
$$
\Lambda(T):= \int_0^T a(x)\, dx = {1 \over p-1} |T|^{p-2} T\,.
$$
With these notations, the diffusion equation can be simply rewritten
under one of the two forms
$$
\partial_t T - \partial_s \left( a(T) \partial_s T  \right)=0\,, \quad \partial_t T - \partial_{ss} \left( \Lambda(T) \right)=0\,.
$$
We shall now introduce the concept of weak solution of problem (\ref{eq1D}) and state the
existence/uniqueness theorem.
\begin{definition}
Let us consider $T^0 \in L^2(\Omega)$ and define $\mathcal{W}\subset L^p(Q)$ as the space 
$$
\mathcal{W}:= \left\{ T \in L^2(Q), \quad |T|^{p-2 \over 2} T \in
L^2_{loc}(\RR^+,H^1(\Omega)), \quad \partial_t T \in
L^{p'}_{loc}(\RR^+,(W^{1,p}(\Omega))^*) \right\},
$$
and we denote by $\mathcal{D}=\C^1_c(\RR^+,W^{1,p}(\Omega))$ the space of test functions. Then the temperature $T \in \mathcal{W}$ is a weak solution to (\ref{eq1D}) if and only it satisfies
\begin{eqnarray}
  &&\int_{\RR^+}\int_{\Omega} T(t,s)\,\partial_t\varphi(t,s) \,ds\,dt-\int_{\RR^+}\int_{\Omega} | T(t,s)|^{p-2}\partial_{s}T(t,s) \partial_{s}\varphi(t,s) dsdt\nonumber\\
 && -\,\gamma\int_{\RR^+} \left[T(t,1)\varphi(t,1)+T(t,0)\varphi(t,0) \right]dt\,+\, \int_{\Omega} T^0(s)\,\varphi(0,s)ds\,=\,0,\quad \forall\varphi\in\mathcal{D}.\nonumber
\label{eqweaksolution}
\end{eqnarray}
\end{definition}

Remark that all the terms in this variational formulation are well-defined. 
Moreover, we observe that a function $T \in L^p(\RR^+\times\Omega)$ satisfying $\partial_t T \in L^{p'}_{loc}(\RR^+,(W^{1,p}(\Omega))^*)$ belongs to $\C([0,\tau],(W^{1,p}(\Omega))^*)$, for all $\tau>0$ by Aubin's Lemma, such that the initial condition is well-defined.

\begin{theorem}
  Let $T^0 \in L^2(\Omega)$ with $T^0 \ge 0$. Then, there exists a unique weak solution $T \in \mathcal{W}$ of
  (\ref{eq1D}), which satisfies $T \in L^{\infty}(\RR^+, L^2(\Omega))$, $T \ge
  0$ almost everywhere and 
$$
{d \over dt} \| T(t,\cdot)\|^2_{L^2(\Omega)} \le 0.
$$
\end{theorem}
The proof of this theorem is decomposed in several steps. For the
beginning, we shall suppose that $T^0 \in L^{\infty}(\Omega)$, with
$\|T^0\|_{\infty} \le M$ and fixed  $M>0$. A truncation can be done, for more general $T^0 \in L^{2}(\Omega)$.\\
Two main difficulties arise in the mathematical study of (\ref{eq1D}),
the nonlinearity and the degeneracy, which means that the equation
changes its type there where $T=0$. 
\begin{proof}
We shall first regularize the problem, in order to avoid the degeneracy. Then, in a
second step, we shall treat the nonlinearity via a fixed point argument. Finally, {\it a priori} estimates shall help us to pass to the limit, in order to deal with the degenerate problem. Let us thus detail these steps.

\paragraph{\bf First step: Regularization.} Let $0< \epsilon<1$ be fixed and let us define the regularized diffusion coefficient
$$
a_{\epsilon,M}(T):= \left[ \epsilon^2 + \min ( |T|^2, M^2) \right]^{p-2 \over 2}\,,
$$
and the corresponding primitive
$$
\Lambda_{\epsilon,M}(T):= \int_0^T a_{\epsilon,M}(x)\, dx \,.
$$
The diffusion coefficients being now bounded from below and above, standard arguments allow to prove that the regularized problem
\begin{equation}
  \left\{
  \begin{array}{l}
 \displaystyle{   \partial_tT_{\epsilon,M}-\partial_s(a_{\epsilon,M}(T_{\epsilon,M}) \partial_sT_{\epsilon,M})\,=\,0,\quad (t,s)\in \RR^+ \times (0,1), }
\\
\,
\\
  \displaystyle{   a_{\epsilon,M}(T_{\epsilon,M}) \partial_sT_{\epsilon,M}\,=\,\gamma T_{\epsilon,M},\quad s\,=\,0,} 
\\
\,
\\
  \displaystyle{   a_{\epsilon,M}(T_{\epsilon,M}) \partial_sT_{\epsilon,M}\,=\,-\gamma T_{\epsilon,M},\quad s\,=\,1,}
\\
\,
\\
 \displaystyle{   T(0, \cdot)\,=\,T^0,}
  \end{array}
\right.
\label{pb_reg}
\end{equation}
has a unique weak solution $T_{\epsilon,M} \in L^2(\RR^+, H^1(\Omega))$ such that it satisfies the following variational formulation: for any $\varphi\in \C^1_c(\RR^+, H^1(\Omega))$
\begin{eqnarray} 
\label{var}
\int_{\RR^+} \int_{\Omega}  T_{\epsilon,M}(t,s) \partial_t\varphi(t,s) \,ds\,dt\,-\,\int_{\RR^+}\int_{\Omega} a_{\epsilon,M}(T_{\epsilon,M}) \partial_{s}T_{\epsilon,M}(t,s) \partial_{s}\varphi(t,s) \,ds\,dt
\\
\nonumber
 -\,\gamma\,\int_{\RR^+} \left[T_{\epsilon,M}(t,1)\phi(t,1)+T_{\epsilon,M}(t,0)\phi(t,0) \right]dt\,+\,\int_{\Omega}  T^0(s)\varphi(0,s) \,ds \,=\,0.
\end{eqnarray}
These arguments are based on the Schauder fixed point theorem, applied
on the mapping $\mathcal{T}$ : $B_R  \mapsto B_R$ with 
$$
B_R \,:=\, \left\{ v \in L^2(Q), \,\, \|v\|_{L^2(Q)} \le R \right\}\,,
$$
where for $v \in B_R$ we associate $\mathcal{T} v$ the solution of the linearized problem associated to (\ref{pb_reg}).

\paragraph{\bf 2nd step: a priori estimates.} In order to pass to the limit $\epsilon \rightarrow 0$, we will need
some {\it a priori} estimates for the solution $T_{\epsilon,M}$, independent of $\epsilon$. Taking in the variational formulation (\ref{var}) as test function $T_{\epsilon,M}$,  yields first
\begin{eqnarray*}
 {1 \over 2} \int_{\Omega}  |T_{\epsilon,M}(t,s)|^2 ds  \,+\, \int_0^t \int_{\Omega}
a_{\epsilon,M}(T_{\epsilon,M}) |\partial_s T_{\epsilon,M}|^2 ds d\tau 
\\
 +\,\gamma\int_0^t \left[|T_{\epsilon,M}(\tau,1)|^2+ |T_{\epsilon,M}(\tau,0)|^2 \right]d\tau\,\,=\,\, {1 \over 2} \int_{\Omega}
 |T^0(s)|^2 ds
\end{eqnarray*}
which implies that for all $t\geq 0$
$$
\left\{
\begin{array}{l}
\displaystyle \|T_{\epsilon,M}(t)\|_{L^2(\Omega)} \,\,\le\,\, \|T^0\|_{L^2(\Omega)}\,, 
\\
\,
\\ 
\displaystyle\int_0^t \int_{\Omega} a_{\epsilon,M}(T_{\epsilon,M}) |\partial_s T_{\epsilon,M}|^2 ds\, d\tau \,\,\le\,\, \|T^0\|^2_{L^2(\Omega)}\,.
\end{array}\right.
$$
This shows also, that the sequence $\{ |T_{\epsilon,M}|^{p-2 \over 2} T_{\epsilon,M} \}_{\epsilon}$ is bounded in $L^2(\RR^+, H^1(\Omega))$ and hence $\{ T_{\epsilon,M} \}_{\epsilon}$ bounded in $L^p(Q)$. Moreover by standard arguments for parabolic problems we deduce than, that $\{ \partial_t T_{\epsilon,M} \}_{\epsilon}$ is bounded in $L^{p'}(\RR^+, (W^{1,p}(\Omega))^*)$.

\paragraph{\bf Third step: passing to the limit.} The {\it a priori} estimates of the last step permit us to show, that there is a sub-sequence and a function $T_M \in L^{2}(Q)$, such that 
$$
T_{\epsilon,M} \rightharpoonup T_M \quad \textrm{in} \quad L^2(\RR^+\times\Omega)\,\quad\textrm{ as } \epsilon \rightarrow 0.
$$
Moreover, from standard compactness arguments \cite{simon}, we show that the following set of measurable functions
$$
\mathcal{F}:= \left\{ u: \RR^+ \rightarrow (W^{1,p}(\Omega))^*,\quad \int_{\RR^+} \int_{\Omega} |u|^{p-2} |\partial_s u|^2 \,ds \, dt \le C\,, \,\, \partial_t u \in L^{p'}(\RR^+,(W^{1,p}(\Omega))^*) \right\}\,,
$$
is compactly embedded in $L^p(\RR^+\times\Omega)$, implying thus that, up to a sub-sequence
$$
T_{\epsilon,M} \rightarrow T_M, \textrm{ in } L^p(\RR^+\times\Omega)\,,\textrm{ as } \epsilon \rightarrow 0
$$
and then $T_{\epsilon,M} \rightarrow T_M$ a.e. in $Q$ when $\epsilon$ goes to zero.  Furthermore, since $\{ |T_{\epsilon,M}|^{p-2 \over 2} T_{\epsilon,M} \}_{\epsilon}$ is
bounded in $L^2(\RR^+, H^1(\Omega))$, one has 
$$
|T_{\epsilon,M}|^{p-2 \over 2} T_{\epsilon,M}
\rightharpoonup |T_{M}|^{p-2 \over 2}T_{M}\,,
\quad \textrm{in} \quad L^2(\RR^+, H^1(\Omega)),\textrm{ as } \epsilon \rightarrow 0,
$$
implying by the weak continuity of the trace application
$$
|T_{\epsilon,M}|^{p-2 \over 2} T_{\epsilon,M}
\rightharpoonup |T_{M}|^{p-2 \over 2}T_{M},\textrm{ in } L^2(\RR^+\times\partial \Omega),\textrm{ as } \epsilon \rightarrow 0.
$$
Finally, we also have using the same arguments, when  $\epsilon \rightarrow 0$
$$
\left\{
\begin{array}{l}
\displaystyle\Lambda_{\epsilon,M}(T_{\epsilon,M})\rightharpoonup \Lambda_{M}(T_{M}), \textrm{ in } L^2(\RR^+, H^1(\Omega)),
\\
\,
\\
\displaystyle\partial_t T_{\epsilon,M}\rightharpoonup \partial_t T_{M}, \textrm{ in } L^{p'}(\RR^+, (W^{1,p}(\Omega))^*).
\end{array}\right.
$$
All these convergences permit us now to pass to the limit in the variational formulation (\ref{var}) in order to show the existence of a weak solution of problem (\ref{eq1D}). This solution is even unique and satisfies the maximum principle, which can be shown as in step 2.
\end{proof}

%

\subsection{A finite volume approximation}
In this section, we propose to derive a numerical scheme for (\ref{eq1D}) in which we apply a finite volume approach for the discretization in the space variable. Let us consider a set of points $(s_{i-1/2})_{0\leq i\leq n_s}$ of the interval $(0,1)$
with $s_{-1/2}=0$, $s_{n_s-1/2}=1$  and $n_s+1$ represents the number of discrete points. For $0\leq i \leq n_s-1$, we define the control cell $C_i$ by the space interval $C_i=(s_{i-1/2},s_{i+1/2})$. We  also denote by $s_i$ the middle of $C_i$  and by $\Delta s_i$ the space step $\Delta s_i = s_{i+1/2}-s_{i-1/2}$ where we suppose that there exists $\xi\in(0,1)$ such that
\begin{equation}
\label{cond:mesh}
\xi\,\Delta s \,\leq \,\Delta s_i \,\leq\,\Delta s, \quad \forall i\in\{0,\ldots,n_s-1\},
\end{equation}
with $\Delta s=\max_i\Delta s_i$.

We shall construct a set of approximations $T_i(t)$ of the average of the solution to (\ref{eq1D}) on the control volume $C_i$ and first set 
$$
T_i^0 = \frac{1}{\Delta s_i}\,\int_{C_i} T_0(s)\,ds.
$$  
Applying a finite volume discretization  to (\ref{eq1D}), $T_i$ is solution to a system of ODEs, which  can be written as 
\begin{equation}
\label{sch:01}
\left\{
\begin{array}{l}  
\displaystyle\frac{dT_i}{dt}(t) \,=\,\frac{\F_{i+1/2}-\F_{i-1/2}}{\Delta s_i}\,, \quad 0\leq i\leq n_s-1,
\\
\,
\\
T_i(t=0) = T_{i}^0, \quad 0\leq i\leq n_s-1,
\end{array}\right.
\end{equation}
where the numerical flux is given by
\begin{equation}
\label{sch:flux}
\F_{i+1/2} \,=\,\frac{4\,K_\parallel}{7}\,\,\frac{\left(T_{i+1}\right)^{7/2}-\left(T_{i}\right)^{7/2}}{\Delta s_{i+1}+\Delta s_i} \,, \quad i=0,\dots,n_s-2.
\end{equation}
Moreover, at the boundary $s=0$ and $s=1$, we apply the boundary conditions, 
\begin{equation}
\label{sch:bc}
\F_{i+1/2}=
\left\{
\begin{array}{ll}
\displaystyle +\gamma \,T_{0}, & \textrm{if } i=-1,
\\
\,
\\
\displaystyle-\gamma \,T_{n_s-1}, &  \textrm{if } i=n_s-1.
\end{array}\right.
\end{equation}
Note that the above discretization on space is first order due to the loss of precision at the boundary. To complete the discretization to the system (\ref{eq1D}), the finite volume scheme (\ref{sch:01})-(\ref{sch:bc}) has to be supplemented with a stable and consistent time discretization step. In the following we present different time discretizations starting from classical explicit and implicit schemes and then propose a stable and accurate numerical approximation.
 
\subsection{Time explicit discretization}
We denote by $\Delta t>0$ the time step, $t^n\,=\,n\Delta t$ for any $n\in\NN$ and $T^n$ is an approximation of the solution $T$ to (\ref{eq1D}) at time $t^n$. Then, we apply a backward Euler scheme to (\ref{sch:01})-(\ref{sch:bc}), which yields
\begin{equation}
\label{sch:02}
\left\{
\begin{array}{l}  
\displaystyle\frac{T_i^{n+1}-T_i^n}{\Delta t} \,=\,\frac{\F_{i+1/2}^n-\F_{i-1/2}^n}{\Delta s_i}\,, \quad 0\leq i\leq n_s-1,
\\
\,
\\
T_i^0 = T_{0,i}, \quad 0\leq i\leq n_s-1,
\end{array}\right.
\end{equation}
with $\F_{i+1/2}^n$ the flux (\ref{sch:flux})-\eqref{sch:bc} computed from the approximation at time $T^n$.

Classically, to guarantee the stability of the scheme (\ref{sch:02}), the time step $\Delta t$ is restricted by a CFL condition. 
\begin{proposition}
\label{prop:01}
Consider that the initial datum $T_0$ is nonnegative and  $T_0\in L^\infty(0,1)$ and assume the stability condition
\begin{equation}
\Delta t\leq \frac{\xi^2\Delta s^2}{\max\left(\frac{4\,K_\parallel}{7}\|T_0\|^{5/2}_\infty,\gamma\Delta s\right)}\,,
\label{dtcfl}
\end{equation}
where $\xi$ is given in (\ref{cond:mesh}). Then, the numerical solution $(T_i^n)_{i,n}$ obtained by the explicit scheme (\ref{sch:02}) is stable and converges to the exact solution to (\ref{eq1D}). 
\end{proposition}
We don't give the proof of this result since it is similar to the proof of Proposition \ref{prop:02} presented in the next section. Unfortunately, this simple scheme is not really efficient since it becomes costly when the mesh is very fine, the constraint on the time step becoming too restrictive. 

\subsection{Time implicit discretization}
To avoid the restrictive constraint on the time step (\ref{dtcfl}), an implicit scheme is more suitable. Therefore, we consider the finite volume scheme  (\ref{sch:01})-(\ref{sch:bc}) to the system of equations (\ref{eq1D}), but apply a forward Euler time discretization. This yields,
\begin{equation}
\label{sch:03}
\left\{
\begin{array}{l}  
\displaystyle\frac{T_i^{n+1}-T_i^n}{\Delta t} \,=\,\frac{\F_{i+1/2}^{n+1}-\F_{i-1/2}^{n+1}}{\Delta s_i}\,, \quad 0\leq i\leq n_s-1,
\\
\,
\\
T_i^0 = T_{0,i}, \quad 0\leq i\leq n_s-1,
\end{array}\right.
\end{equation}
with $\F_{i+1/2}^{n+1}$ the flux (\ref{sch:flux}) computed from the approximation at time $T^{n+1}$.  Hence, a fully nonlinear system has to be solved at each time step.

The scheme (\ref{sch:03}) coupled with (\ref{sch:flux})-(\ref{sch:bc}) is uniformly stable and leads to a numerical approximation which converges to the exact solution to (\ref{eq1D}). 
\begin{theorem}
\label{theo:01}
Consider that the initial datum $T_0$ is nonnegative and $T_0\in L^\infty(0,1)$. Then the numerical solution given by the implicit scheme (\ref{sch:03}) coupled with (\ref{sch:flux})-(\ref{sch:bc}) is uniformly stable in $L^\infty(\RR^+\times (0,1))$ and converges to the weak solution $T$ of (\ref{eq1D}) when $h=(\Delta t,\Delta s)$ goes to zero.
\end{theorem}
We start with a stability result and then prove  convergence of the numerical solution to the unique weak solution by consistency of the scheme.

Let us first investigate the stability property and prove some {\it a priori} estimates on the numerical solution uniformly with respect to the mesh size $h$.
\begin{proposition}
\label{prop:02}
Consider that the initial datum $T_0$ is nonnegative and $T_0\in L^\infty(0,1)$. Then the numerical solution given by the implicit scheme (\ref{sch:03}) coupled with (\ref{sch:flux})-(\ref{sch:bc}) is unconditionally stable, {\it i.e.}
\begin{equation}
\label{bound:linf}
0\leq T_i^n \leq \|T_0\|_{L^\infty},
\end{equation}
and
\begin{equation}
\label{bound:l2}
\sum_{i=0}^{n_s-1} \Delta s_i \,|T_i^{n+1}|^2 \,\,\leq \,\,\sum_{i=0}^{n_s-1} \Delta s_i \,|T_i^0|^2. 
\end{equation}
Moreover, the following discrete semi-norm is uniformly bounded
\begin{equation}
\label{bound:derivee}
\sum_{n=0}^{N_t}\sum_{i=0}^{n_s-2}\Delta t \frac{\left[\left(T^{n+1}_{i+1}\right)^{7/2}-\left(T^{n+1}_{i}\right)^{7/2}\right]^2}{\Delta s_i + \Delta s_{i+1}}\,\leq\, C,
\end{equation}
where the constant $C>0$ only depends on the initial datum $T_0$.
\end{proposition}
\begin{proof}
Let us consider a convex function $\phi\in \C^1(\RR,\RR)$,  then we have
\begin{eqnarray}
\phi(T^{n+1}_i) - \phi(T_i^n)  \,\leq\, \phi'(T^{n+1}_i)(T^{n+1}_i-T^{n}_i).
\label{eqconvexdiscret}
\end{eqnarray}
Thus, we multiply the scheme (\ref{sch:03}) by $\Delta t\,\Delta s_i\,\phi^\prime(T_i^{n+1})$ and sum over $i\in\{0,\dots,n_s-1\}$, it gives 
\begin{eqnarray*}
\sum_{i=0}^{n_s-1}\Delta s_i\,\phi(T^{n+1}_i)\,-\,\sum_{i=0}^{n_s-1}\Delta s_i\,\phi(T^{n}_i) &\leq& \Delta t\,\sum_{i=0}^{n_s-1}\phi'(T^{n+1}_i)\,\left( \F_{i+1/2}^{n+1}-\F_{i-1/2}^{n+1} \right),
\\
&\leq& -\,\Delta t\,\sum_{i=0}^{n_s-2}\F_{i+1/2}^{n+1}\,\left( \phi'(T^{n+1}_{i+1})\,-\, \phi'(T^{n+1}_{i})\right)
\\
&-& \Delta t\,\F_{-1/2}^{n+1}\, \phi'(T^{n+1}_{0}) \,+\, \Delta t\,\F_{n_s-1/2}^{n+1}\, \phi'(T^{n+1}_{n_s-1}).
\end{eqnarray*}
Using the definition of the numerical flux (\ref{sch:flux}) and the discrete boundary conditions (\ref{sch:bc}), we get 
\begin{eqnarray*}
\sum_{i=0}^{n_s-1}\Delta s_i\,\phi(T^{n+1}_i)\,-\,\sum_{i=0}^{n_s-1}\Delta s_i\,\phi(T^{n}_i) \,\leq\, -\gamma\,\Delta t\,\phi'(T^{n+1}_0)T^{n+1}_0\,-\,\gamma\,\Delta t\,\phi'(T^{n+1}_{n_s-1})T^{n+1}_{n_s-1}
\\
-\,\frac{4\,K_\parallel}{7}\,\Delta t\;\sum_{i=0}^{n_s-2}\left[\phi'(T^{n+1}_{i+1})-\phi'(T^{n+1}_{i})\right]\,\frac{\left(T^{n+1}_{i+1}\right)^{7/2}-\left(T^{n+1}_{i}\right)^{7/2}}{\Delta s_i+\Delta s_{i+1}}.
\end{eqnarray*}
Observing that a similar inequality holds true when  $\phi(x)$ is only Lipschitzian, we take $\phi(x)=x^-$, and prove the nonnegativity of the approximation $T_i^n$, that is,
    \begin{equation*}
      0\leq\,\sum_{i=0}^{n_s-1}\Delta s_i \,(T^{n+1}_i)^-\,\leq\,\sum_{i=0}^{n_s-1}\Delta s_i\,(T^{0}_i)^-\,=\,0.
    \end{equation*}
Therefore, assuming that $T^0_i\geq0 $, for all $0\leq i \leq n_s-1$, we obtain that $T^n_i\geq0$ for all $0\leq i\leq n_s -1$ and $n\in\NN$. Moreover, taking $\phi(x)=(x-M)^+$, with $M=\|T^0\|_{L^\infty}$, we have
\begin{equation*}
0\leq \sum_{i=0}^{n_s-1}\Delta s_i\,(T^{n+1}_i-M)^+\,\leq\,\sum_{i=0}^{n_s-1}\Delta s_i\,(T^{n}_i-M)^+\,\leq\,\sum_{i=0}^{n_s-1}\Delta s_i\,(T^{0}_i-M)^+=0.
\end{equation*}
Hence we deduce that  $0\leq T^n_i\leq M$, for all $0\leq i \leq n_s-1$.

Then we take $\phi(x)=x^2/2$, which yields that
$$
\sum_{i=0}^{n_s-1}\frac{\Delta s_i}{2}\,(T^{n+1}_i)^2\,-\,\sum_{i=0}^{n_s-1}\frac{\Delta s_i}{2}\,(T^{n}_i)^2 \,\leq\,C\,\Delta t\sum_{i=0}^{n_s-2}\left[T^{n+1}_{i+1}-T^{n+1}_{i}\right]\,\frac{\left(T^{n+1}_{i+1}\right)^{7/2}-\left(T^{n+1}_{i}\right)^{7/2}}{\Delta s_i+\Delta s_{i+1}}
$$ 
and use the fact that $T_i^n$ is uniformly bounded to observe that
$$
|T_{i+1}^{7/2} \,-\,T_{i}^{7/2}|\,\leq \,C\,|T_{i+1}\,-\,T_i|.
$$ 
Thus, we have the following inequality
\begin{eqnarray*}
  \frac{1}{2}\sum_{i=0}^{n_s-1}\Delta s_i\,(T^{n+1}_i)^2&\leq& \frac{1}{2}\sum_{i=0}^{n_s-1}\Delta s_i\,(T^{n}_i)^2
\\
&-&C\Delta t\,\sum_{i=0}^{n_s-2}\frac{\left[\left(T^{n+1}_{i+1}\right)^{7/2}-\left(T^{n+1}_{i}\right)^{7/2}\right]^2}{\Delta s_i + \Delta s_{i+1}}.
\end{eqnarray*}
Finally we sum over  $n\in\{0,\dots,N_t\}$ and immediately deduce that there exists a constant $C>0$ only depending on the initial datum  $T_0$ such that  
$$
\Delta t\,\sum_{n=0}^{N_t}\sum_{i=1}^{n_s-1}\frac{\left[\left(T^{n+1}_{i}\right)^{7/2}-\left(T^{n+1}_{i-1}\right)^{7/2}\right]^2}{\Delta s_i + \Delta s_{i+1}}\,\leq\, C.
$$
\end{proof}
\subsection{Proof of Theorem~\ref{theo:01}}
To prove the convergence of the discrete solution $(T^n_i)_{i,n}$ towards the weak solution $T$ to (\ref{eq1D}), we construct  a  piecewise approximation $T_h$, where $h=(\Delta t,\Delta s)$,  such that 
$$
T_h(t,s):=\sum_{n\in\NN}\sum_{i=0}^{n_s-1} T_i^{n}\,\mathbf{1}_{C_i} (s)\, \mathbf{1}_{[t^{n}, t^{n+1}[} (t),
$$
From the uniform bounds proved in Proposition~\ref{prop:02}, we get that there exits a sub-sequence, still denoted by $(T_h)_h$, such that $T_h$ converges to $T\in L^\infty(\RR^+\times(0,1))$ as $m\to\infty$ in the  weak-* topology, whereas using (\ref{bound:derivee}) we also get that $T_h^{7/2}$ converges strongly in $L^2(\RR^+\times(0,1)$ to $T_h^{7/2}$.

Now let us prove that $T_h$ converges to the weak solution to (\ref{eq1D}) when $h$ goes to zero. We consider $\varphi\in \C^\infty_c(\RR^+\times(0,1))$, and we denote $\varphi_i^n=\varphi(t^n,s_i)$. Then we multiply the  scheme (\ref{sch:03}) by  $\Delta s_i\varphi_i^n$, and sum over $i\in\{0,\ldots,n_s-1\}$ and $n\in\NN$, we obtain
\begin{equation*}
  \E^1_h\,+\,\E^2_h\,=\,0,
\end{equation*}
with $\E^1_h$ is related to the time discretization and is given by 
\begin{eqnarray*}
  \E^1_h\,:=\,\sum_{n\in\NN}\sum_{i=0}^{n_s-1}\Delta s_i\,(T_i^{n+1}\,-\,T_i^{n})\,\varphi_i^n,
 \end{eqnarray*}
whereas $\E_h^2$ is related to the space discretization and reads 
\begin{eqnarray*}
\E^2_h\,:=\,\Delta t\sum_{n\in\NN}\sum_{i=0}^{n_s-1} \left(\F_{i+1/2}^{n+1} - \F_{i-1/2}^{n+1}\right)\,\varphi_i^n.
\end{eqnarray*}
On the one hand, we consider $\E^1_h$ and perform a discrete integration by part with respect to $n\in\NN$. Using that $\varphi$ is compactly supported for large $t\in\RR^+$, it yields
\begin{eqnarray*}
  \E^1_h &=&-\sum_{n\in\NN^*}\sum_{i=0}^{n_s-1}\Delta s_i\,T_i^n\,(\varphi_i^{n}-\varphi_i^{n-1}) \,-\,\sum_{i=0}^{n_s-1}\Delta s_i\,T_i^{0}\,\varphi_i^{0}.
\\ 
&=&- \int_{\RR^+}\int_0^1 T_h(t+\Delta t,s)\,\partial_t\varphi(t,s)dsdt -\int_0^1 T_h(0,s)\,\varphi(0,s)ds \,\,+\, <\mu_h^1,\varphi>   
\end{eqnarray*}
where the additional term $<\mu_h^1,\varphi>$ is given by 
\begin{eqnarray*}
<\mu_h^1,\varphi>  &=& -\sum_{n\in\NN^*}\sum_{i=0}^{n_s-1}\int_{t^{n-1}}^{t^n}\int_{C_i}\int_s^{s_i}T_i^n\,\partial_{ts}^2\varphi(t,\eta)d\eta dsdt \\
&-& \sum_{i=0}^{n_s-1} \int_{C_i}\int_s^{s_i}T_i^0\,\partial_s\varphi(0,\eta)d\eta\,ds
\end{eqnarray*}
and satisfies the following estimate
\begin{eqnarray*}
|<\mu_h^1,\varphi>|\,\leq\, C\,\Delta s\,\left(\,\|\partial_{ts}^2\varphi\|_{L^1} \,+\,  \|\partial_{s}\varphi(0,.)\|_{L^1} \,\right).
\end{eqnarray*}
Therefore, when $h$ tends to zero, we have
$$
\E_h^1 \rightarrow - \int_{\RR^+}\int_0^1 T(t,s)\,\partial_t\varphi(t,s)dsdt -\int_0^1 T_0(s)\,\varphi(0,s)ds.
$$
 On the other hand, we apply a first discrete integration by part with respect to $i\in\{0,\ldots,n_s-1\}$ to the second term $\E_h^2$, which can be written as
\begin{eqnarray*}
  \E_h^2 &=& -\frac{4\,K_\parallel}{7}\,\Delta t\sum_{n\in\NN}\sum_{i=0}^{n_s-2} \frac{\left(T_{i+1}^{n+1}\right)^{7/2}-\left(T_{i}^{n+1}\right)^{7/2}}{\Delta s_{i+1}+\Delta s_i}\,\left(\varphi_{i+1}^n \,-\, \varphi_i^n  \right)
\\
&-& \gamma\,\Delta t\sum_{n\in\NN}  T_{0}^{n+1}\,\varphi_{0}^n \,-\,  \gamma\,\Delta t\sum_{n\in\NN} T_{n_s-1}^{n+1}\,\varphi_{n_s-1}^n.
\end{eqnarray*}
Then, introducing $D_h T_h$ a discrete  approximation of the gradient of $T_h$ by  
$$
D_hT_h(t,s):=\sum_{n\in\NN}\sum_{i=0}^{n_s-1} 2\,\frac{T_{i+1}^{n}-T_i^n}{\Delta s_i+\Delta s_{i+1}}\,\mathbf{1}_{(s_i,s_{i+1})} (s)\, \mathbf{1}_{[t^{n}, t^{n+1}[} (t),
$$
we have 
\begin{eqnarray*}
\E_h^2 &=& \frac{2\,K_\parallel}{7}\,\int_{\RR^+}\int_{s_0}^{s_{n_s-1}} D_{h} T_h^{7/2}(t+\Delta t,s)\,\partial_s\varphi(t,s)dsdt 
\\
&-& \gamma\,\int_{\RR^+} T_h(t+\Delta t,0)\varphi(t,s_0) +   T_h(t+\Delta t,1)\varphi(t,s_{n_s-1})dt
\end{eqnarray*}
Passing to the limit $h\rightarrow 0$, we  get that
$$
\E_h^2\rightarrow \frac{2\,K_\parallel}{7}\,\int_{\RR^+}\int_{0}^{1} \partial_sT^{7/2}(t,s)\,\partial_s\varphi(t,s)dsdt - \gamma\,\int_{\RR^+} T(t,0)\varphi(t,0) +   T(t,1)\varphi(t,1)dt.
$$ 
Finally, we conclude that $T$ is a weak solution of (\ref{eq1D}). By uniqueness of the solution to (\ref{eq1D}), it yields that the sequence $(T_h)_{h}$ converges to the weak solution of (\ref{eq1D}).

The implicit scheme (\ref{sch:03}) is unconditionally stable, but it requires the numerical resolution of a nonlinear system. For this purpose a Newton method is applied which increases considerably the computational cost and makes this method inefficient. Another strategy would consist in applying a semi-implicit scheme for the time discretization, but it still requires the implementation of a new linear system at each time iteration and the computational cost remains too important.  In the following we propose a numerical scheme inspired by the work of F. Filbet \& S. Jin \cite{bibAclass} to handle with this problem.

\subsection{An implicit-explicit (IMEX) scheme}
In \cite{bibAclass}, the authors proposed to handle with a stiff and nonlinear problem. The main point is to write the nonlinear problem in a different form  in order to split the nonlinear operator in the sum of a dissipative linear part, which can be solved in an implicit way and a non dissipative and nonlinear part which will be solved with a time explicit solver. The main difficulty is to find an adequate decomposition of the operator. For instance the nonlinear diffusive operator can be written as
$$
K_\parallel\partial_{s}\left(T^{5/2}\partial_{s}T\right)  \,\,=\,\, \nu\,\partial_{ss}^2 T \,+\,\partial_{s}\left(\left( K_\parallel\,T^{5/2}\,-\,\nu\right)\,\partial_s T\right)
$$
and the time discretization to (\ref{eq1D}) becomes
\begin{equation}
\label{sch:04}
 \left\{
\begin{array}{l}
\displaystyle\frac{ T^{n+1}-T^n}{\Delta t}-\partial_s\left(\nu\partial_sT^{n+1}\right)=\partial_s\left(\left(K_\parallel\left(T^n\right)^{5/2}-\nu\right)\partial_sT^{n}\right),
\\
\,
\\
\displaystyle-\nu\partial_sT^{n+1}(0)+\gamma T^{n+1}(0)=\left(K_\parallel \left(T^n(0)\right)^{5/2}-\nu\right)\partial_sT^{n}(0),
\\
\,
\\
 \displaystyle -\nu\,\partial_sT^{n+1}(1)-\gamma T^{n+1}(1)=\left(K_\parallel \left(T^n(1)\right)^{5/2}-\nu\right)\partial_sT^{n}(1). 
\end{array}\right.
\end{equation}
To choose an appropriate $\nu$ for the scheme (\ref{sch:04}), we perform an energy estimate of the numerical approximation. 
\begin{proposition}
\label{prop:03}
Assume that the viscosity term $\nu$ is such that  
\begin{equation}
\label{choix:nu}
K_\parallel\left\|T^n\right\|_\infty^{5/2}\,\leq\,\nu,\quad \forall n\in\NN.
\end{equation}
Then the numerical solution satisfies the following 
\begin{equation}
\label{esti:sch:04}
 \frac{1}{2}\int_0^1\left(T^{n+1}\right)^2ds\,+\,\frac{\nu\Delta t}{2}\,\int_0^1\left|\partial_sT^{n+1}\right|^2ds \,\leq\,
\frac{1}{2}\int_0^1\left(T^{n}\right)^2ds \,+\, \frac{\nu\Delta t}{2}\,\int_0^1\left|\partial_sT^{n}\right|^2ds.
\end{equation}
\end{proposition}
\begin{proof}
We multiply (\ref{sch:04}) by $T^{n+1}$ and integrate on $s\in(0,1)$, hence we have
\begin{eqnarray*}
\frac{1}{2}\int_0^1\left|T^{n+1}\right|^2ds-\frac{1}{2}\int_0^1\left|T^{n}\right|^2ds &\leq&\int_0^1\left(\left(T^{n+1}\right)^2- T^{n+1}T^n\right)ds
\\
&\leq&\Delta t\int^1_0\left(\left(\nu-K_\parallel\left|T^n\right|^{5/2}\right)\partial_s T^{n}\partial_s T^{n+1}-\nu\left(\partial_s T^{n+1}\right)^2\right)ds
\\
&-&\Delta t\gamma\left(\left(T^{n+1}_0\right)^2+\left(T^{n+1}_{n_s-1}\right)^2\right).
\end{eqnarray*}
Using the assumption that $K_\parallel\left|T^n\right|^{5/2}\leq\nu$ and applying the Young's inequality, we obtain
\begin{eqnarray*}
\left(\nu\,-\,\left|T^n\right|^{5/2}\right)\,\partial_s T^{n}\,\partial_s T^{n+1}&\leq&\frac{\varepsilon}{2}\,\left(\partial_sT^n\right)^2\,+\,\frac{\left(\nu\,-\,K_\parallel\,\left|T^n\right|^{5/2}\right)^2}{2\,\varepsilon}\,\left(\partial_sT^{n+1}\right)^2
\\
&\leq&\frac{\varepsilon}{2}\,\left(\partial_sT^n\right)^2\,+\,\frac{\nu^2}{2\varepsilon}\,\left(\partial_sT^{n+1}\right)^2.
\end{eqnarray*}
Therefore with the choice $\varepsilon=\nu$, we have
\begin{equation*}
 \frac{1}{2}\int_0^1\left(T^{n+1}\right)^2ds\,+\,\frac{\nu}{2}\,\Delta t\,\int_0^1\left|\partial_sT^{n+1}\right|^2ds \,\leq\,
\frac{1}{2}\int_0^1\left(T^{n}\right)^2ds \,+\, \frac{\nu}{2}\,\Delta t\int_0^1\left|\partial_sT^{n}\right|^2ds.
\end{equation*}
Hence, the scheme (\ref{sch:04}) is stable when $K_\parallel\|T^n\|_\infty^{5/2}\leq\nu$. 
\end{proof}

Now, we can give the fully discrete scheme, called in the sequel IMEX, as follows
\begin{equation}
\label{sch:04bis}
\left\{
\begin{array}{l}  
\displaystyle\frac{T_i^{n+1}-T_i^n}{\Delta t} \,=\,\frac{\F_{i+1/2}^{n+1/2}-\F_{i-1/2}^{n+1/2}}{\Delta s_i}\,, \quad 0\leq i\leq n_s-1,
\\
\,
\\
T_i^0 = T_{0,i}, \quad 0\leq i\leq n_s-1,
\end{array}\right.
\end{equation}
with the numerical flux $\F_{i+1/2}^{n+1/2}$ is given for $i\in\{0,\dots,n_s-2\}$ by
\begin{equation}
\F_{i+1/2}^{n+1/2} \,=\,2\,\left(K_\parallel\,\left(\frac{\left(T_{i+1}^n\right)^{5/2}+\left(T_{i}^n\right)^{5/2}}{2}\right)\,-\,\nu\right) \,\frac{T_{i+1}^n-T_{i}^n}{\Delta s_{i+1}+\Delta s_i} \,+\, 2\,\nu\,\frac{T_{i+1}^{n+1}-T_{i}^{n+1}}{\Delta s_{i+1}+\Delta s_i} \,, 
\label{grenouille:0}
\end{equation}
whereas at the boundary $s=0$ and $s=1$, we apply the boundary conditions written in the form (\ref{sch:04}), 
\begin{equation}
\label{grenouille:1}
\F_{i+1/2}^{n+1/2}=
\left\{
\begin{array}{ll}
\displaystyle +\gamma \,T_{0}^{n+1}, & \textrm{if } i=-1,
\\
\,
\\
\displaystyle-\gamma \,T_{n_{s}-1}^{n+1}, &  \textrm{if } i=n_s-1.
\end{array}\right.
\end{equation}
Moreover, the viscosity $\nu>0$ is initially chosen as an upper bound of $K_\parallel \|T^0\|_{\infty}^{5/2}$ and is then readjusted along iterations $n\in\NN$ in order to satisfy the condition (\ref{choix:nu}):

\begin{algorithmic}
\STATE \,
\STATE{\bf Algorithm to compute $\nu$}
\STATE \,
\STATE $\nu \,:=\,2\,K_\parallel \|T^0\|_{\infty}^{5/2}$ and  $n=0$

\WHILE{$n \leq N_{T_{end}}$}

\STATE compute the numerical solution $T^{n+1}$
\STATE\,
\IF {$\nu\leq \frac{5}{4}\,K_\parallel \|T^{n+1}\|_{\infty}^{5/2}$} 
\STATE $\nu\gets 2\,K_\parallel \|T^{n+1}\|_{\infty}^{5/2}$
\ENDIF
\STATE\,
\IF {$\nu\geq 4\,K_\parallel \|T^{n+1}\|_{\infty}^{5/2}$} 
\STATE $\nu\gets K_\parallel \|T^{n+1}\|_{\infty}^{5/2}\,/\,2$
\ENDIF
\STATE\,
\STATE  $n\gets n+1$
\ENDWHILE
\end{algorithmic}

\subsection{Numerical results}
To compare the numerical results obtained with the different schemes, we take $\gamma=2$, $K_\parallel=1$ and the initial temperature is $  T^0=5$, whereas the final time of the numerical simulation is equal to $T_{end}=1$. On the one hand a reference solution is computed using the finite volume method with an explicit scheme (\ref{sch:02}) on a uniform grid with $n_s=450$. On the other hand, we basically compare both implicit (\ref{sch:03}) and IMEX (\ref{sch:04bis})-(\ref{grenouille:1}) schemes with different uniform grids with  $n_s=50$, $150$. Furthermore, we choose the time step equal to $\Delta t=10^{-2}$, $10^{-3}$, $10^{-4}$ and $10^{-5}$ respectively. 

\definecolor{gris25}{gray}{0.75}
\begin{table}[H]%
  \begin{center}
  \begin{tabular}{|l|l|l|l|l|l|}
    \hline
    \multicolumn{2}{|c|}{$\Delta t$}&$10^{-2}$&$10^{-3}$&$10^{-4}$&$10^{-5}$   \\
     \hline
     \hline
   \multirow {2}*{Implicit scheme (\ref{sch:03})}&$n_s=50$&0.05&0.31&2.24&22.06\\
   &\cellcolor{gris25}$n_s=150$&\cellcolor{gris25}0.60&\cellcolor{gris25}4.07&\cellcolor{gris25}27.49&\cellcolor{gris25}249.61\\
 \hline
 \multirow {2}*{IMEX scheme (\ref{sch:04bis})-(\ref{grenouille:1})}&$n_s=50$&0.01&0.09&0.63&5.34\\
&\cellcolor{gris25}$n_s=150$&\cellcolor{gris25}0.10&\cellcolor{gris25}0.24&\cellcolor{gris25}2.24&\cellcolor{gris25}21.97\\
 \hline
  \end{tabular}
\end{center}

\caption {\label{table1}Computational time for the implicit scheme (\ref{sch:03}) and the IMEX scheme (\ref{sch:04bis})-(\ref{grenouille:1}) in seconds at the final time of the numerical simulation $T_{end}=1$.}
\end{table}
We observe from Table~\ref{table1} that the IMEX scheme is much more efficient than the implicit scheme in terms of computational cost since the linear system corresponding to the implicit part does not depend on the iteration $n$ when the viscosity $\nu>0$ is large enough. For $n_s=50$, the computational time of the IMEX scheme is less than one fourth  of the one corresponding to the implicit scheme whereas for $n_s=150$, the implicit scheme is ten times more consuming than IMEX scheme. 
\begin{table}[H]%
  \begin{center}
  \begin{tabular}{|l|l|l|l|l|l|}
    \hline
   \multicolumn {2}{|c|} {$\Delta t$}&$10^{-2}$&$10^{-3}$&$10^{-4}$&$10^{-5}$   \\
     \hline\hline
   \multirow {2}*{Implicit scheme (\ref{sch:03})}&$n_s=50$&0.0580&0.0612&0.0617&0.0617\\
&\cellcolor{gris25}$n_s=150$&\cellcolor{gris25}0.0190&\cellcolor{gris25}0.0184&\cellcolor{gris25}0.0187&\cellcolor{gris25}0.0187\\
 \hline
 \multirow {2}*{IMEX scheme (\ref{sch:04bis})-(\ref{grenouille:1})}&$n_s=50$&0.0621&0.0600&0.0598&0.0598\\
&\cellcolor{gris25}$n_s=150$&\cellcolor{gris25}0.0213&\cellcolor{gris25}0.0182&\cellcolor{gris25}0.0181&\cellcolor{gris25}0.0181\\
 \hline
  \end{tabular}
\end{center}
\caption {\label{table2}Relative errors  obtained using an implicit scheme, IMEX scheme at time  $T_{end}=1$.}

\end{table}

Concerning the accuracy and stability, Table~\ref{table2} shows that the numerical solution computed with both implicit and IMEX schemes is stable for any time step $\Delta t$ and the numerical errors are of the same order. Moreover, we get similar results when time step is smaller than $10^{-4}$. Of course, when we increase the number of points $n_s$, the numerical error decreases and the IMEX scheme (\ref{sch:04bis})-(\ref{grenouille:1}) seems to be more accurate for small time steps. Finally, in Figure~\ref{figexplicit}{\color{red}a}, we observe that the large errors appear around the boundary, where large gradients of temperature occur. The Figure~\ref{figexplicit}{\color{red}b} illustrates the temperature evolution at different time $t=0.25,\,0.50,\,0.75$ and $1$. We note that the temperature has a fast decay at the beginning, then it stabilizes to a steady state when $t$ approaches the final time $T_{end}=1$. Furthermore we observe that the temperature develops steep gradients  at the boundary modeling the cooling of the plasma due to the limiter effects. Indeed, on the one hand the thermal diffusion depends on the term $T^{5/2}$ which is large at the beginning  and then becomes smaller and smaller. On the other hand, due to the nonlinear flux at the boundary when the temperature becomes small, the temperature gradient becomes larger and larger.

\begin{center}
\begin{figure}[htbp]

\begin{tabular}{cc}
  \includegraphics[width=7cm]{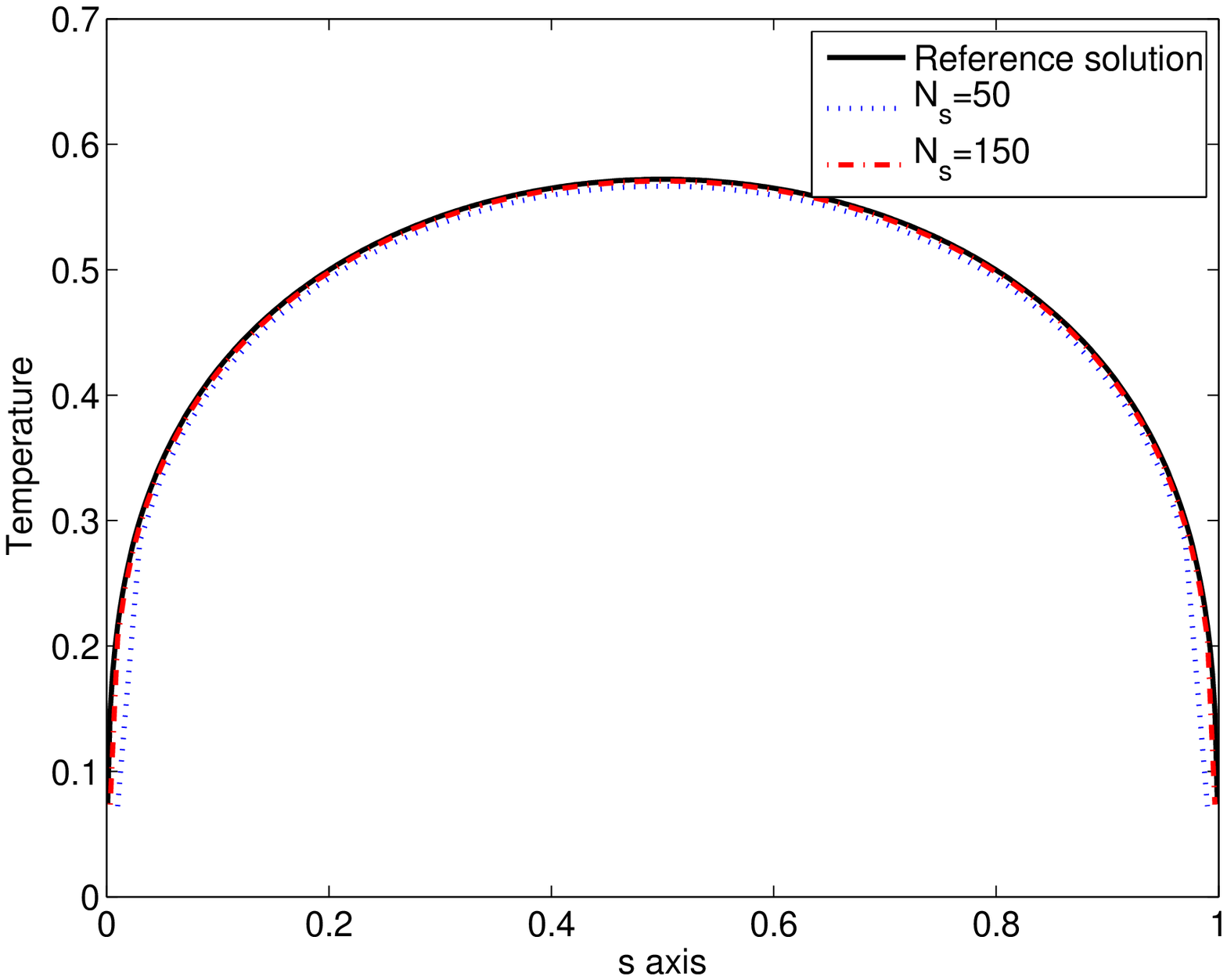} &   
   \includegraphics[width=7cm]{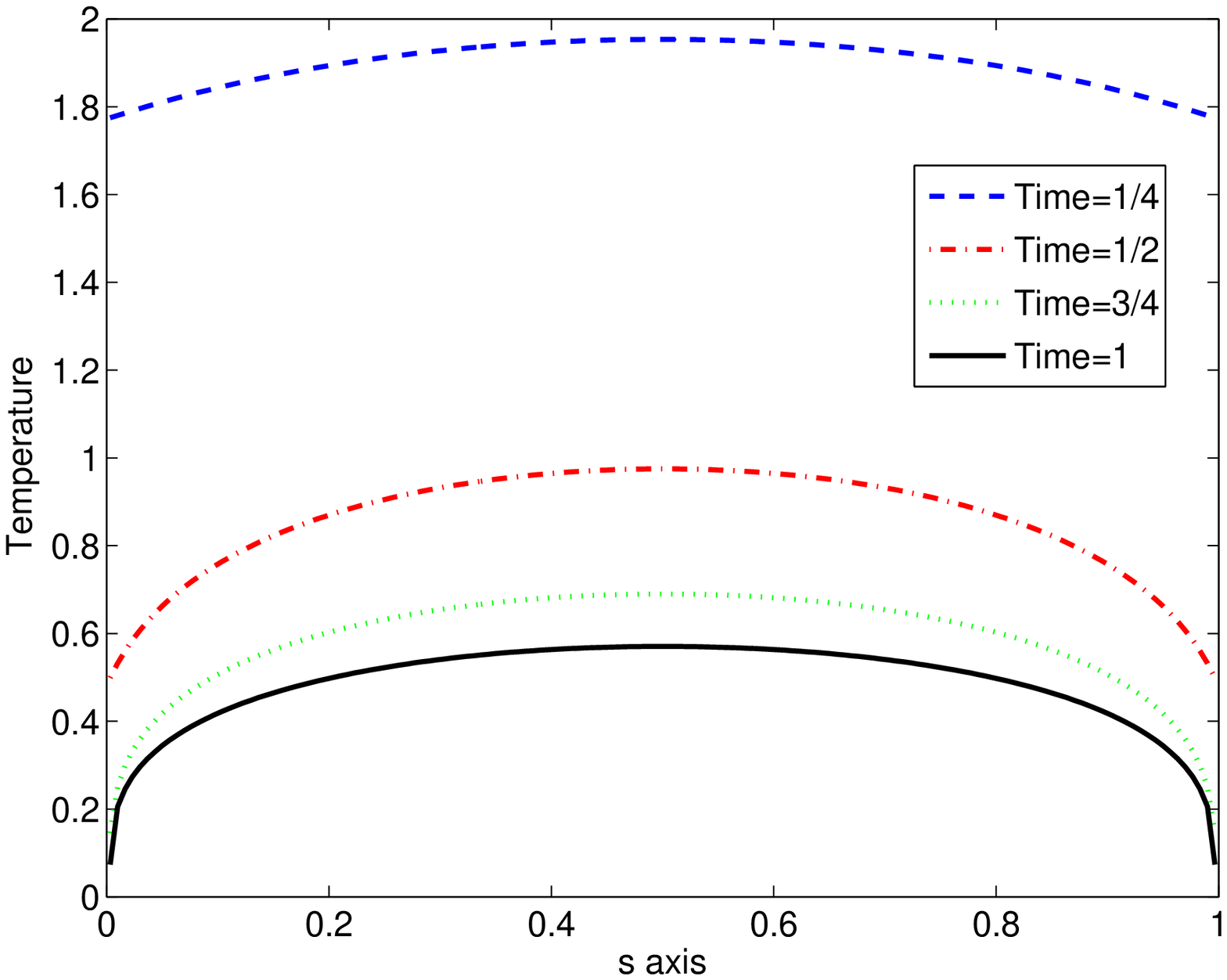}
\\
(a) & (b) 
\end{tabular}
  \caption {\label{figexplicit}Temperature evolution of problem~\eqref{eq1D}. We use the IMEX scheme to approximate~\eqref{eq1D} and choose time step of $\Delta t=10^{-4}$. (a) reference solution and the results of IMEX scheme for $n_s=50$, $150$ at time $T_{end}=1$, (b) results of IMEX scheme for $n_s=150$ at time $t=0.25,\,0.50,\,0.75$ and $1$.}
\end{figure}
\end{center}

\section{\bf The 2D problem}
\label{sec:3}
\setcounter{equation}{0}

In this section, we consider  the two dimensional problem where the temperature $T$ depends on time $t$ and two space variables $(s,r)\in \Omega=(0,1)\times (0,1)$ with appropriate boundary conditions 
\begin{equation}
\partial_tT\,-\,\partial_s(K_{\parallel}\,T^{5/2}\,\partial_sT)\,-\,\partial_r(K_{\bot}\,\partial_rT)\,=\,0,\quad t\geq 0,\,(s,r)\in\Omega,
\label{eq2D}
\end{equation}
where $K_\parallel$ and $K_\perp$ are nonnegative constants with $K_{\perp}\ll K_\parallel$. For the boundary conditions we impose a boundary flux in $r=0$ and assume that for $r=1$ the flux of temperature is zero, that is,
\begin{equation}
  \left\{
  \begin{array}{l}
    \partial_rT(t,s,0)\,=\,-Q_{\bot},\quad s\in(0,1),\, r=0,\,t\geq 0,
\\
\,
\\
    \partial_rT(t,s,1)\,=\,0, \quad s\in(0,1),\, r=1,\,t\geq 0,
\end{array}\right.
\label{eq2Dbcr}
\end{equation}
and at the boundary $s=0$ and $s=1$ we consider either periodic boundary conditions or of modelling describing the effects of the limiter which allows to decrease the temperature in the device. At $s=0$, we have
\begin{equation}
  \left\{
  \begin{array}{l}
K_{\parallel}T^{5/2}(t,0,r)\,\partial_sT(t,0,r)\,=\,\gamma \,T(t,0,r),\quad r\in(1/2,1),\,t\geq 0,
\\
\,
\\
 T(t,0,r) \,=\, T(t,1,r), \quad r\in(0,1/2),\,t\geq 0,
  \end{array}
\right.
\label{eq2Dbcs0}
\end{equation}
and $s=1$,
\begin{equation}
  \left\{
  \begin{array}{l}
K_{\parallel}\,T^{5/2}(t,1,r)\,\partial_sT(t,1,r)\,=\,-\gamma\, T(t,1,r),\quad r\in(1/2,1),\,t\geq 0,
\\
\,
\\
 T(t,0,r) \,=\, T(t,1,r), \quad r\in(0,1/2),\,t\geq 0.
  \end{array}
\right.
\label{eq2Dbcs1}
\end{equation}

This model also satisfies an energy estimate given by 
\begin{eqnarray*}
\frac{1}{2} \frac{d}{dt}\int_{\Omega} |T(t,s,r)|^2dsdr &=& -\frac{16\,K_\parallel}{81}\int_{\Omega} |\partial_s T^{9/4}|^2 dsdr\,-\, K_\perp\int_{\Omega} |\partial_r T|^2 dsdr 
\\
&-& \gamma \int_{1/2}^1 \left(T(0,r)+T(1,r)\right) dr \,+\,  K_\perp \,Q_\perp\,\int_{0}^1 T(s,0)ds.
\end{eqnarray*}
To discretize the system (\ref{eq2D})-(\ref{eq2Dbcs1}), we apply a finite volume method in space $\Omega$ coupled with a time splitting scheme for the time discretization. We first present the numerical scheme and describe precisely the discretization of the boundary conditions. Finally we  compare our numerical results with those obtained by standard explicit and implicit time discretizations.

\subsection{Time splitting scheme}\label{sec:timesplitting_2D}
We apply a time splitting scheme in both directions. As for the one dimensional case, we apply an IMEX scheme to treat the nonlinear equation and find a condition on the viscosity $\nu>0$ to get a uniformly stable scheme. We first consider the non linear problem in the $s$ direction,
\begin{equation}
\label{sch:101}
      \frac{T^{\star}-T^n}{\Delta t}\,-\,\partial_s\left(\left(K_\parallel (T^n)^{5/2}\;-\,\nu\right)\partial_sT^n\right)\,-\,\nu\partial_{ss}^2T^{\star}\,=\, 0,\quad (s,r)\in\Omega, 
\end{equation}
with the boundary condition (\ref{eq2Dbcs0}), 
\begin{equation}
\label{sch:102}
\left\{
\begin{array}{l}
\left(K_\parallel \left(T^n(0,r)\right)^{5/2}\,-\,\nu\right)\;\partial_sT^{n}(0,r) \,=\,\gamma T^{\star}(0,r)\,-\,\nu\,\partial_sT^{\star}(0,r), \quad r\in (1/2,1),
\\
\,
\\
T^\star(0,r) = T^\star(1,r), \quad r\in (0,1/2), 
\end{array}\right.
\end{equation}
and then the condition (\ref{eq2Dbcs1}),
\begin{equation}
\label{sch:103}
\left\{
\begin{array}{l}
\left(K_\parallel \left(T^n(1,r)\right)^{5/2}\,-\,\nu\right)\;\partial_sT^{n}(1,r)\,=\,-\,\gamma T^{\star}(1,r)\,-\,\nu\,\partial_sT^{\star}(1,r),\quad r\in (1/2,1),
\\
\,
\\
T^\star(1,r) = T^\star(0,r), \quad r\in (0,1/2), 
\end{array}\right.
\end{equation}
which allows to compute a first approximation $T^\star$. Then we compute a numerical approximation of the linear heat equation,
\begin{equation}
\label{sch:11}
\frac{T^{n+1}-T^\star}{\Delta t}\,-\,\partial_r(K_\perp\partial_rT^{n+1})\,=\,0,\quad (s,r)\in\Omega, 
\end{equation}
with  non homogeneous Neumann boundary conditions
\begin{equation}
\label{sch:12}
 \left\{
  \begin{array}{l}
    \partial_rT^{n+1}(s,0)\,=\,-Q_{\bot},\quad s\in(0,1),\, r=0,
\\
\,
\\
    \partial_rT^{n+1}(s,1)\,=\,0, \quad s\in(0,1),\, r=1.
\end{array}\right.
\end{equation}

For the sake of clarity we present a stability estimate on this semi-discrete scheme (discrete in time and continuous in space), but the proof can be easily adapted to the fully discrete case.
\begin{proposition}\label{prop:2D}
Assume that the viscosity term $\nu$ is such that  for any $r\in(0,1)$,
\begin{equation*}
K_\parallel\left\|T^n\right\|_\infty^{5/2}\leq\,\nu,\quad \forall n\in\NN.
\end{equation*}
Then the numerical solution satisfies the following 
\begin{eqnarray*}
\frac{1}{2}\int_{\Omega}\left(T^{n+1}\right)^2\,dr\,ds + \frac{\nu\Delta t}{2}\,\int_{\Omega}|\partial_sT^{n+1}|^2 \,dr\,ds
\,\leq\,\frac{1}{2}\int_{\Omega}\left(T^0\right)^2\,dr\,ds + \frac{\nu\Delta t}{2} \int_{\Omega}|\partial_sT^0|^2dr\,ds 
\\
 -\,K_\perp\,\Delta t\,\sum_{k=1}^{n+1}\left[\int_{\Omega}|\partial_rT^{k}|^2\,dr\,ds \,-\, Q_\perp\,\int_{0}^1 T^{k}(s,0) ds\right].
\end{eqnarray*}
\label{prop:04}
\end{proposition}
\begin{proof}
 Multiplying (\ref{sch:101}) by $T^\star$ and integrating in $\Omega$, we obtain
\begin{eqnarray*}
\frac{1}{2}\int_{\Omega}\left(\left(T^\star\right)^2-\left(T^n\right)^2\right)\,dr\,ds &\leq&-\Delta t\,\int_\Omega\left(K_\parallel\left(T^n\right)^{5/2} - \nu\right)\partial_sT^n\,\partial_sT^\star\,dr\,ds 
\\
&&-  \Delta t\;\int_\Omega \,\nu|\partial_sT^\star|^2\,dr\,ds
\\
&&-  \gamma\,\Delta t\;\int_{1/2}^1 |T^\star(0,r)|^2\,+\,|T^\star(1,r)|^2\,dr.
\end{eqnarray*}
Then, applying the Young inequality and taking $\nu$ such that for all $r\in (0,1)$,
$$
0\,\leq\,K_\parallel\left|T^n(s,r)\right|^{5/2}\leq\,\nu,\quad \forall n\in\NN,
$$
we have
\begin{equation}
\label{rama:01}
\frac{1}{2}\int_{\Omega}\left(T^\star\right)^2\,dr\,ds \,+\; \frac{\nu\Delta t}{2}\,\int_{\Omega}|\partial_sT^\star|^2dr\,ds  \,\leq\, \frac{1}{2}\int_{\Omega}\left(T^n\right)^2\,dr\,ds \,+\, \frac{\nu\Delta t}{2}\,\int_{\Omega}|\partial_sT^n|^2dr\,ds.
\end{equation}
Similarly, we multiply (\ref{sch:11}) by $ T^{n+1}$ and integrate with respect to $(s,r)\in \Omega$, we get
\begin{eqnarray}
\frac{1}{2}\int_{\Omega}\left(\left(T^{n+1}\right)^2-\left(T^\star\right)^2\right)drds&\leq&-\Delta t \,K_\perp\,\int_\Omega\left(\partial_r T^{n+1}\right)^2drds
\nonumber
\\
&& +\,\Delta t \,Q_\perp \,K_\perp\,\int^1_0T^{n+1}(s,0)ds.
\label{rama:02}
\end{eqnarray}
Furthermore, we derive (\ref{sch:11}) with respect to $s$ and get
\begin{equation*}
\frac{\partial_sT^{n+1}-\partial_sT^{\star}}{\Delta t}-K_\perp(\partial_{rr}^2\partial_sT^{n+1})=0.
\end{equation*}
Then we multiply this latter equality by $\nu\partial_sT^{n+1}$ and integrate over $(s,r) \in\Omega$, 
\begin{eqnarray*}
\int_{\Omega}\nu\,\left[\left(\partial_sT^{n+1}\right)^2-\left(\partial_sT^\star\right)^2\right]dr\,ds &\leq& -2\,\Delta t\,\nu\,K_\perp\int_{\Omega}|\partial_{rs}T^{n+1}|^2dr\,ds
\\
&+& \nu\,\Delta t\,\left[\partial_s(\partial_rT^{n+1}) \partial_sT^{n+1}\right]_{r=0}^{r=1}.
\end{eqnarray*}
Hence using that $\partial_s\left(\partial_rT^{n+1}(s,r)\right)=0$, $r\in\{0, 1\}$, it yields
\begin{equation}
\label{rama:03}
\int_{\Omega}\nu\left[\left(\partial_sT^{n+1}\right)^2-\left(\partial_sT^\star\right)^2\right]dr\,ds\,\leq\,0.
\end{equation}
Then, gathering (\ref{rama:02}) and (\ref{rama:03}), we get
\begin{eqnarray*}
\frac{1}{2}\int_{\Omega}\left(T^{n+1}\right)^2\,dr\,ds \,+\, \frac{\nu\Delta t}{2}\,\int_{\Omega}\left[|\partial_sT^{n+1}|^2 \,+\,K_\perp\,|\partial_rT^{n+1}|^2 \right]\,dr\,ds \,-\, \Delta t\,K_\perp\,Q_\perp\,\int_{0}^1 T^{n+1}(s,0) ds
\\
\leq\, \frac{1}{2}\int_{\Omega}\left(T^\star\right)^2\,dr\,ds + \frac{\nu\Delta t}{2}\int_{\Omega}|\partial_sT^\star|^2dr\,ds. 
\end{eqnarray*}
Finally, the latter inequality together with (\ref{rama:01}), it gives 
\begin{eqnarray*}
\frac{1}{2}\int_{\Omega}\left(T^{n+1}\right)^2\,dr\,ds \,+\,\frac{\nu\Delta t}{2}\,\int_{\Omega}\left[|\partial_sT^{n+1}|^2 +K_\perp\,|\partial_rT^{n+1}|^2 \right]\,dr\,ds \,-\, \Delta t\,K_\perp\,Q_\perp\,\int_{0}^1 T^{n+1}(s,0) ds
\\
\leq\, \frac{1}{2}\int_{\Omega}\left(T^n\right)^2\,dr\,ds + \frac{\nu\Delta t}{2}\int_{\Omega}|\partial_sT^n|^2dr\,ds. 
\end{eqnarray*}
By induction and summing over $k=0,\ldots,n$, we get the result
\begin{eqnarray*}
\frac{1}{2}\int_{\Omega}\left(T^{n+1}\right)^2\,dr\,ds + \frac{\nu\Delta t}{2}\,\int_{\Omega}|\partial_sT^{n+1}|^2 \,dr\,ds
\,\leq\,\frac{1}{2}\int_{\Omega}\left(T^0\right)^2\,dr\,ds +  \frac{\nu\Delta t}{2}\int_{\Omega}|\partial_sT^0|^2dr\,ds 
\\
 -\,K_\perp\,\Delta t\,\sum_{k=1}^{n+1}\left[\int_{\Omega}|\partial_rT^{k}|^2\,dr\,ds \,-\, Q_\perp\,\int_{0}^1 T^{k}(s,0) ds\right].
\end{eqnarray*}
\end{proof}

\subsection{A finite volume approximation}\label{sec:discretspace2D}

For the space discretization, we consider a set of points $(s_{i-1/2})_{0\leq i\leq n_s}$ a set of points of the interval $(0,1)$ with $s_{-1/2}=0$, $s_{n_s-1/2}=1$ and $n_s+1$ represents the number of discrete points in the direction $s$  and $(r_{j-1/2})_{0\leq j\leq n_r}$ a set of points of the interval $(0,1)$ with $r_{-1/2}=0$, $r_{n_r-1/2}=1$ and $n_r+1$ represents the number of discrete points in the direction $r$. For $0\leq i \leq n_s-1$, $0\leq j \leq n_r-1$,  we define the control cell $C_{i,j}$ by  $C_{i,j}=(s_{i-1/2},s_{i+1/2})\times(r_{j-1/2},r_{j+1/2})$. We  also denote by $(r_i,s_i)$ the center of $C_{i,j}$  and by $\Delta s_i$ the space step $\Delta s_i = s_{i+1/2}-s_{i-1/2}$  and $\Delta r_j$ the space step $\Delta r_j = r_{j+1/2}-r_{j-1/2}$ where we assume that there exists $\xi\in(0,1)$ such that
\begin{equation}
\label{cond:mesh2}
\xi\,h \,\leq \,\Delta s_i,\, \,\Delta r_j\leq\,h, \quad \forall (i,j)\in\{0,\ldots,n_s-1\}\times\{0,\ldots,n_r-1\},
\end{equation}
with $h=\max_{i,j}\{\Delta s_i,\,\Delta r_j\}$.

We shall construct a set of approximations $T_{i,j}(t)$ of the average of the solution to (\ref{temp_simpl})-(\ref{BC}) on the control volume $C_{i,j}$ and set 
$$
T_{i,j}^0 = \frac{1}{|C_{i,j}|}\,\int_{C_{i,j}} T_0(s,r)\,ds\,dr.
$$  
Hence, the finite volume discretization  to (\ref{sch:101}) can be written as 
\begin{equation*}
\frac{T_{i,j}^\star - T_{i,j}^n}{\Delta t} \,=\,\frac{\F_{i+1/2,j}^{n+1/2}-\F_{i-1/2,j}^{n+1/2}}{\Delta s_i}, \quad \forall (i,j)\in\{0,\ldots,n_s-1\}\times\{0,\ldots,n_r-1\},
\end{equation*}
where the flux $\F_{i+1/2,j}$ corresponds to the one dimensional flux given by (\ref{grenouille:0}) and  periodic boundary conditions are applied for $r_j\in(0,1/2)$ and conditions (\ref{grenouille:1}) for  $r_j\in(1/2,1)$. 

Then, the  finite volume discretization  to (\ref{sch:101}) can be written as 
\begin{equation*}
\frac{T_{i,j}^{n+1} - T_{i,j}^\star}{\Delta t} \,=\,\frac{\G_{i,j+1/2}^{n+1}-\G_{i,j-1/2}^{n+1}}{\Delta r_j},\quad \forall (i,j)\in\{0,\ldots,n_s-1\}\times\{0,\ldots,n_r-1\}
\end{equation*}
where $\G_{i,j+1/2}$ is given by
\begin{equation}
\label{sch:flux2}
\G_{i,j+1/2} \,=\,2\,K_\perp\,\frac{T_{i,j+1}^{n+1}\,-\,T_{i,j}^{n+1}}{\Delta r_{j+1}+\Delta r_j} \,, \quad j=0,\dots,n_r-2.
\end{equation}
Moreover, at the boundary $r=0$ and $r=1$, we apply the boundary conditions, 
\begin{equation}
\label{sch:bc2}
\G_{i,j+1/2}=
\left\{
\begin{array}{ll}
\displaystyle -K_\perp\,Q_\perp, & \textrm{if } j=-1,
\\
\,
\\
\displaystyle 0, &  \textrm{if } j=n_r-1.
\end{array}\right.
\end{equation}

\subsection{Numerical results}
In this section we compare the different numerical results related to the 2D problem (\ref{eq2D})-(\ref{eq2Dbcs1}) obtained using a time splitting scheme with an explicit, implicit and IMEX treatment of each step. As before, we first compute a reference solutions obtained from an explicit scheme with a small time step satisfying a CFL condition $\Delta t\sim h^2$.  In the following numerical simulations, we choose  the different physical parameters as $K_\parallel=1$, $K_\perp=10^{-2}$, $\gamma=2$, $Q_\perp=10$. Moreover, the initial temperature is given by
\begin{equation}
 T^0(s,r)\,=3,
\end{equation}
and the final time of the simulation is $T_{end}=2$.

To compute the reference solution, we have chosen $n_s=300$ and $n_{r}=300$, whereas the numerical results using  implicit and IMEX schemes  are obtained with $n_s=100$ and $n_{r}=100$ with several time steps $\Delta t=10^{-1}$, $10^{-2}$, $10^{-3}$, and $10^{-4}$. First, concerning the computational time we observe in Table \ref{tablecomputationtime2D},  that the IMEX scheme is much faster than the implicit scheme. Furthermore, the numerical error presented in  Table \ref{tableerror2D} for both scheme is of the same order of magnitude and thus the IMEX scheme is clearly much more efficient than the fully implicit scheme.
\begin{table}[H]%
  \begin{center}
  \begin{tabular}{|l|l|l|l|l|}
    \hline
    $\Delta t$ & $10^{-1}$ & $10^{-2}$ & $10^{-3}$ & $10^{-4}$\\
     \hline\hline
    Implicit scheme&4.02&25.64&172.95&1327.50\\
    \hline
    IMEX scheme&1.62&4.42&36.24&403.63\\
 \hline
  \end{tabular}
\end{center}
 
\caption {\label{tablecomputationtime2D}Computational time for the 2D problem (\ref{eq2D})-(\ref{eq2Dbcs1}) using implicit  and IMEX schemes at time $T_{end}=2$.}
\end{table}

\begin{table}[H]%
  \begin{center}
  \begin{tabular}{|l|l|l|l|l|}
    \hline
    $\Delta t$& $10^{-1}$ & $10^{-2}$ & $10^{-3}$ & $10^{-4}$\\
     \hline\hline
    Implicit scheme&0.2245&0.0236&0.0020&2.1985e-04\\
    \hline
    IMEX scheme&0.2093&0.0213&0.0018&2.4385e-04\\
 \hline
  \end{tabular}
\end{center}
\caption {\label{tableerror2D}The relative errors for the implicit and IMEX schemes compared with a reference solution for the 2D problem (\ref{eq2D})-(\ref{eq2Dbcs1}) at time  $T_{end}=2$.}
\end{table}

Now we want to investigate the effect of the splitting scheme on the numerical error and the computational cost. Therefore, we also propose a comparison between the different schemes. We first compare the computational time applying  the IMEX scheme with and without the splitting method with a time step $\Delta t=10^{-3}$, $(n_s,n_r)=(50, 50)$, $(100,100)$, $(300,300)$ and $(500,500)$ respectively. On the one hand, we observe in Table \ref{tablesplittingtime} that the splitting method is much faster than the non-splitting method when the number of discrete points increases.
\begin{table}[H]%
  \begin{center}
  \begin{tabular}{|l|l|l|l|l|}
    \hline
    $n_s\times n_r$&$50\times 50$&$100\times100$&$300\times300$&$500\times500$\\
     \hline\hline
    IMEX Non-splitting scheme & 11 & 60 & 505 & 2112\\
    \hline
    IMEX splitting scheme  & 16  & 36 &219 &601 \\
 \hline
  \end{tabular}
\end{center}
\caption {\label{tablesplittingtime}Computational time of IMEX with and without splitting scheme  at time  $T_{end}=2$.}
\end{table}

On the other hand, we compare the numerical  errors corresponding to the two strategies with $(n_s,n_r)=(100,100)$, $\Delta t=10^{-3}$ in Table \ref{tablesplittingerror}, in particular the fully implicit scheme with and without splitting and the IMEX scheme with and without splitting.  We observe that the method without splitting is always more accurate than the one with the splitting method.
\begin{table}[H]%
  \begin{center}
  \begin{tabular}{|l|l|l|l|l|}
    \hline
    Scheme  &Splitting implicit &Splitting IMEX & Implicit & IMEX \\
     \hline\hline
     Numerical error & $2.\times 10^{-3}$ & $2.\times 10^{-3}$ & $5.\times 10^{-4}$ & $5.\times 10^{-4}$ \\
 \hline
  \end{tabular}
\end{center}
\caption {\label{tablesplittingerror}Relative  errors for different numerical schemes compared with a reference solution for  $(n_s,n_r)=(100,100)$, $\Delta t=10^{-3}$ at time  $T_{end}=2$.}
\end{table}
In Figure \ref{figeq2D}, we present  the evolution of the approximation of the temperature (\ref{eq2D})-(\ref{eq2Dbcs1}) in computational domain $\Omega$, which is divided into two regions : the transition layer and the scrape-off layer (SOL) as illustrated in Figure~\ref{ima1}. We first initialize the temperature to a constant and then observe immediately that temperature decreases rapidly in the scrape-off layer  and becomes singular around the limiter (which corresponds to the boundary $s=0$ and $1$ with $r\geq 1/2$). On the other hand, in the transition layer, the temperature converges to a  steady state which is homogeneous in $s\in(0,1)$. The different numerical schemes give the same qualitative behavior of the solution.

\begin{center}
\begin{figure}[htbp]
\begin{tabular}{cc}
\includegraphics[width=7.5cm]{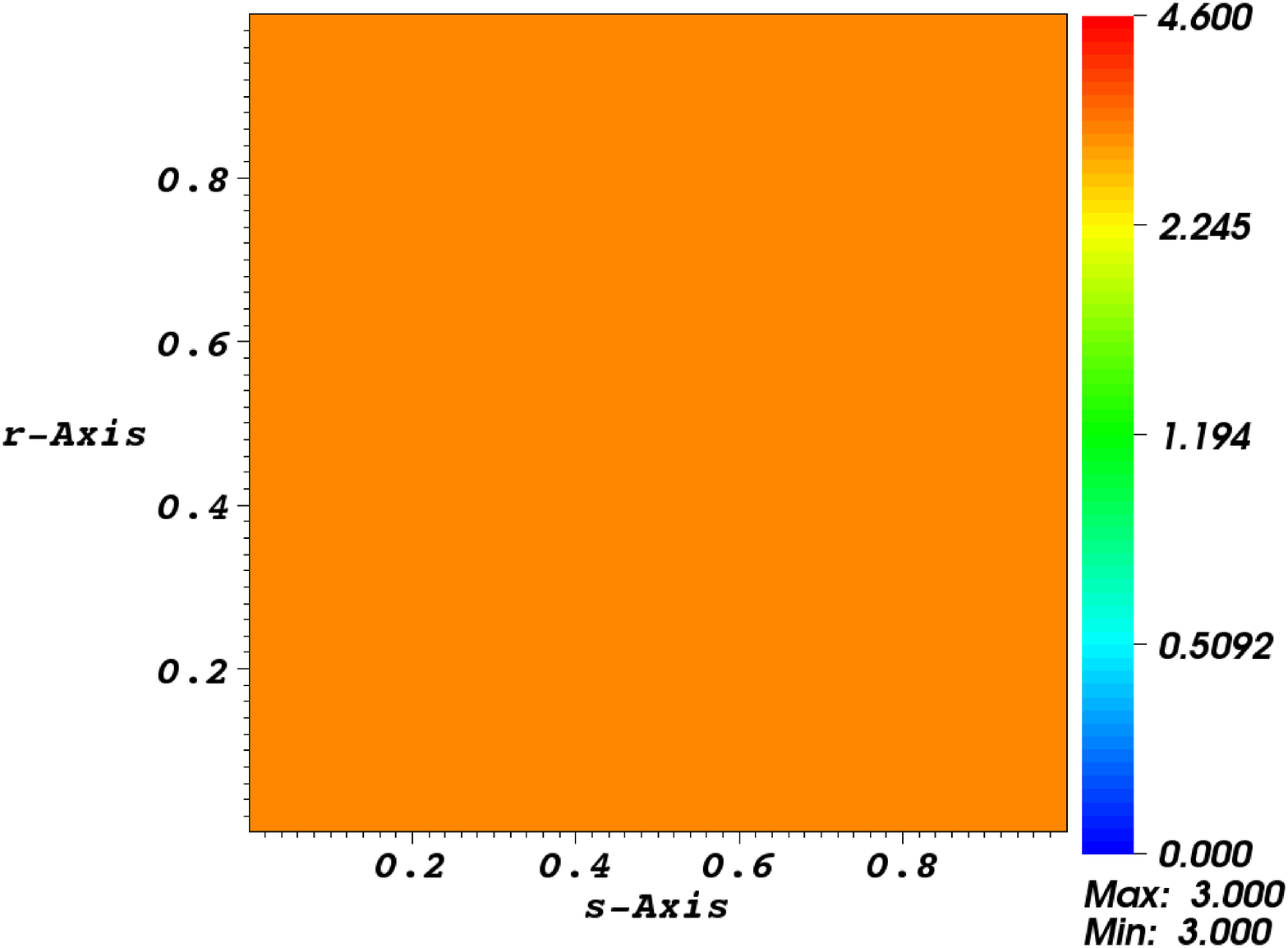}
&
\includegraphics[width=7.5cm]{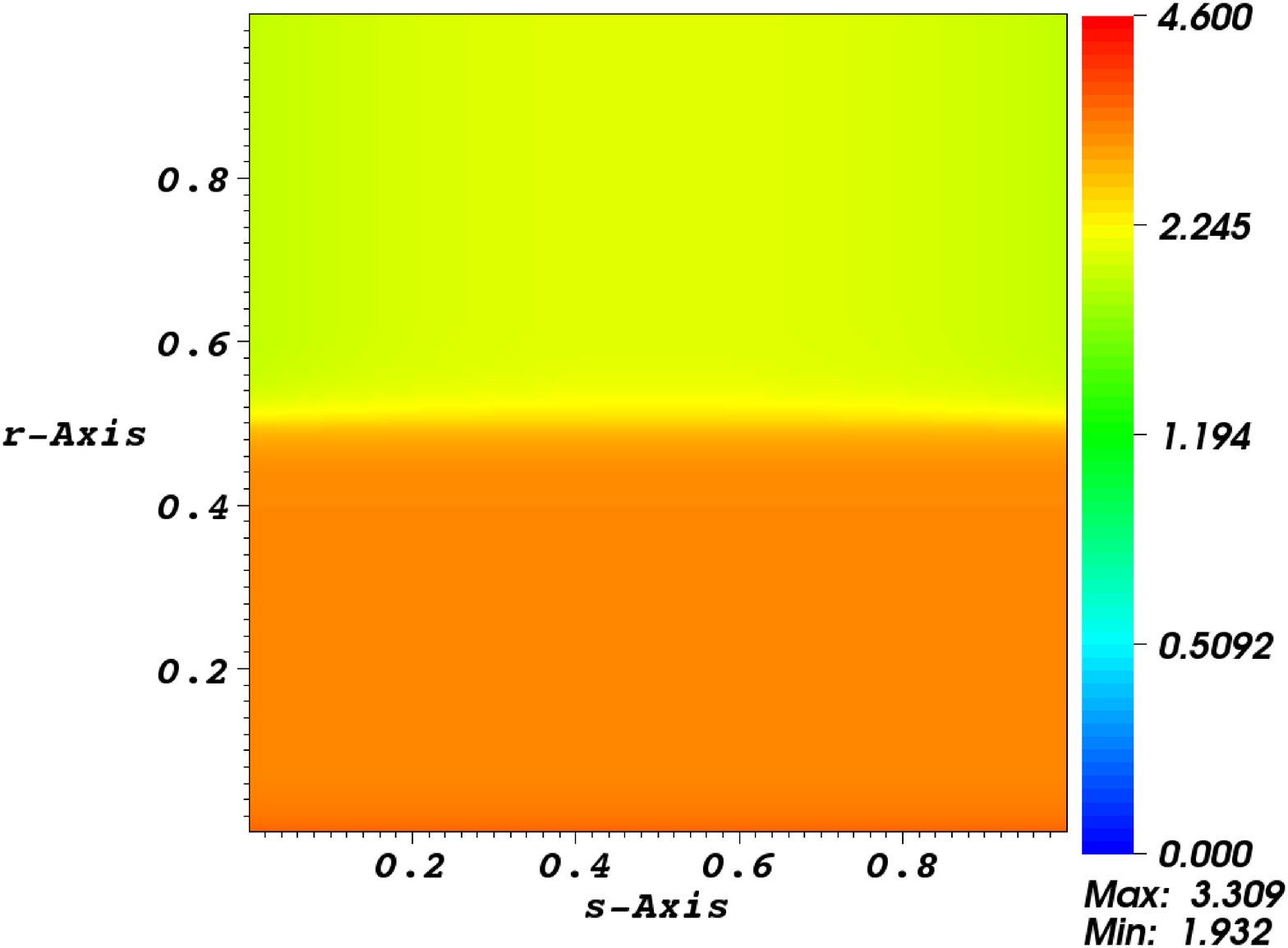}
\\
 (a) $t=0$ & (b) $t=0.1$
\\
 \includegraphics[width=7.5cm]{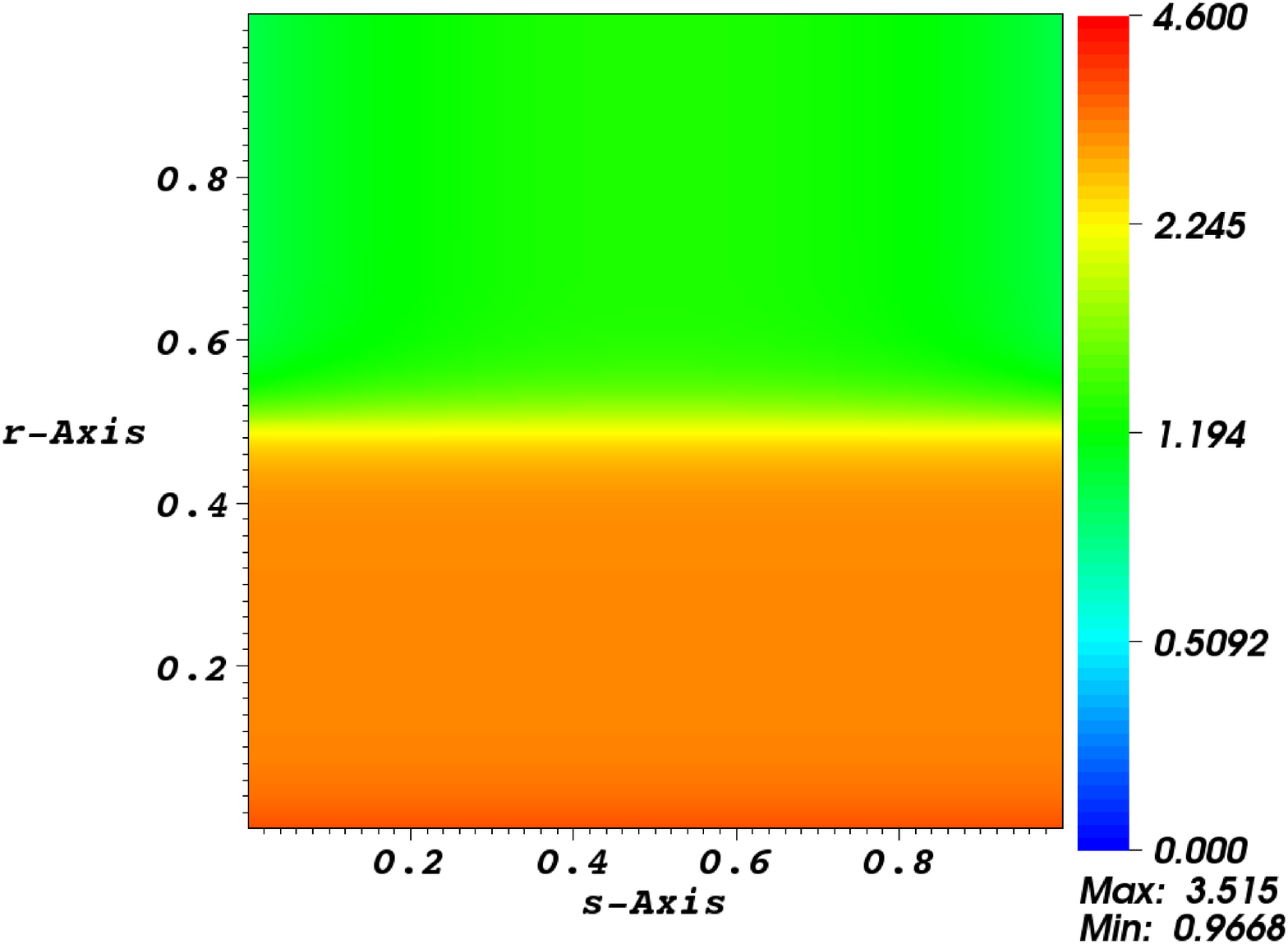}
&
\includegraphics[width=7.5cm]{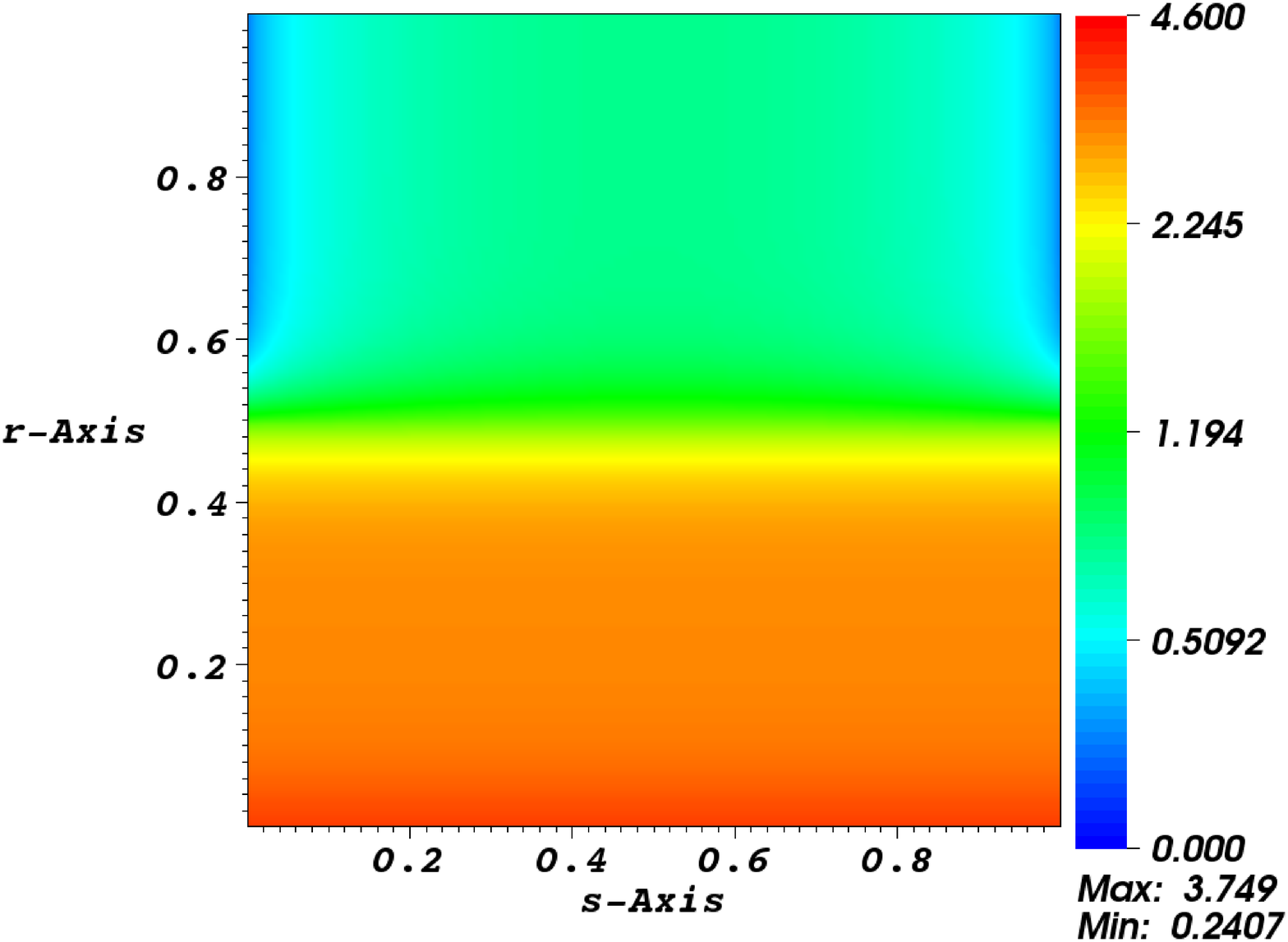}
\\
(c) $t=0.25$
&
(d) $t=0.5$
\\
\includegraphics[width=7.5cm]{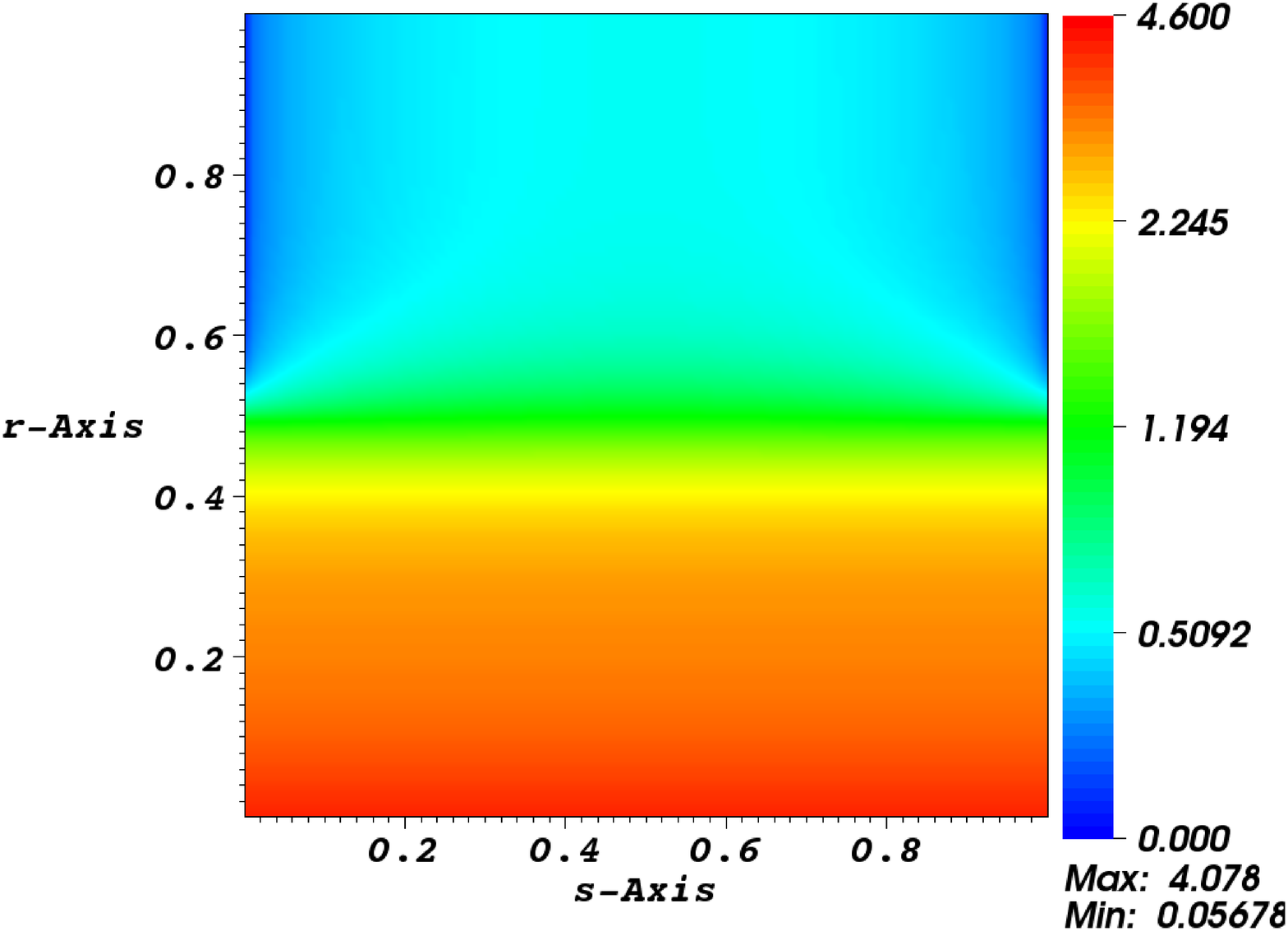} &
\includegraphics[width=7.5cm]{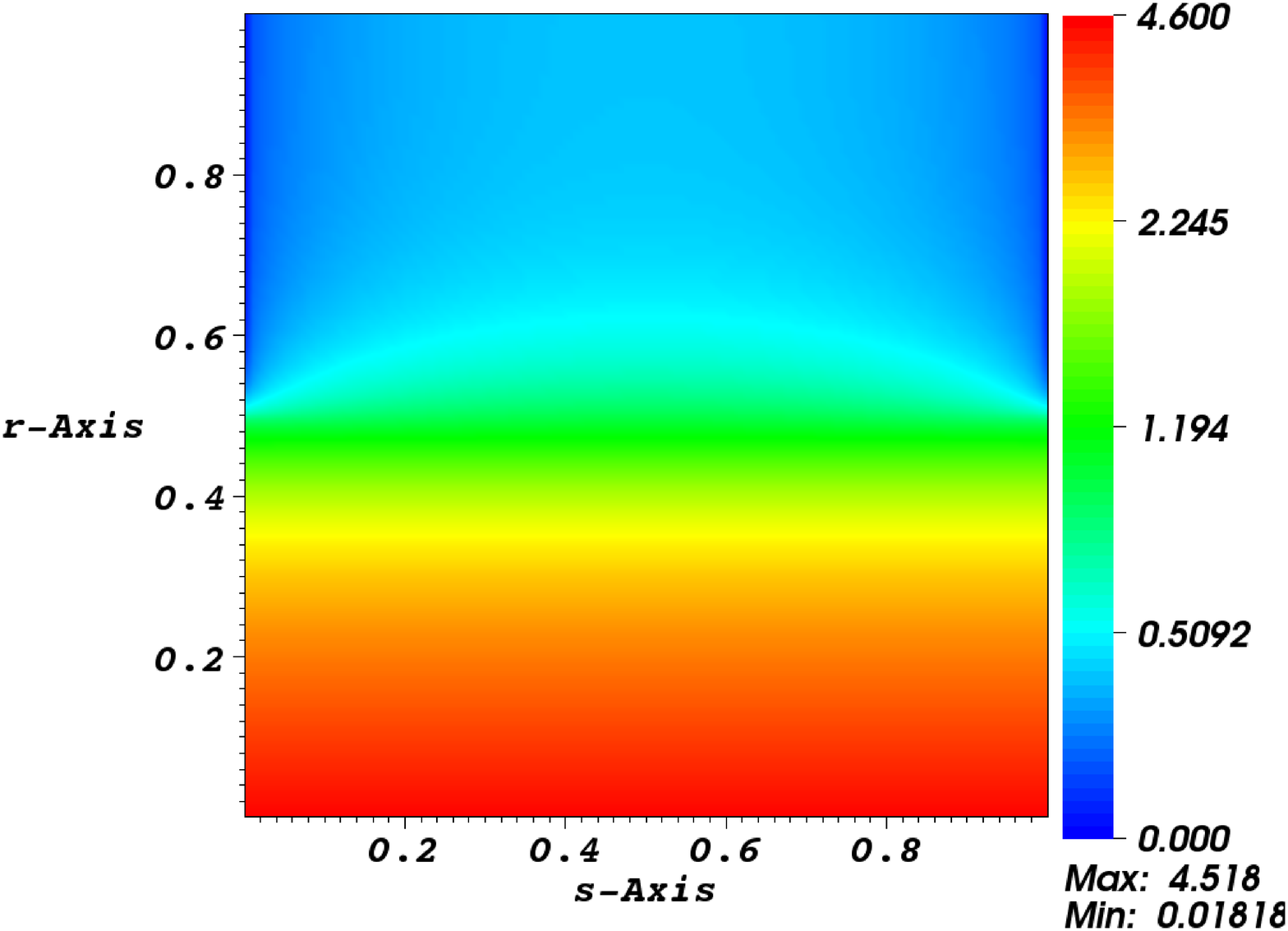}
\\
(e) $t=1$ & (f) $t=2$
\end{tabular}
\caption{Temperature evolution  of problem (\ref{eq2D}).}
  \label{figeq2D}
\end{figure}
\end{center}

In Figure~\ref{fig2Dcoupes}, we plot  the temperature evolution at  the section $r=0.25$, $r=0.75$ and $s=10^{-2}$ and $s=0.5$ respectively. According to Ko\v can {\it et al.}~\cite{Kocan2008,Kocan2009}, the parallel thermal diffusivity is much larger than the  perpendicular one, {\it i.e.} $K_\parallel\gg K_\bot$. Therefore, the temperature becomes constant along the magnetic field lines, that is for $s\in(0,1)$. We observe in Figures~\ref{fig2Dcoupes} that the temperature is constant at all time whereas steep gradients develop at the boundary layer $s=0$ and $s=1)$ in the SOL region. In the perpendicular direction $r$, the situation is different. We also observe that at time $t=2$ the temperature decreases linearly with respect to $r$  in the transition layer ($0\leq r\leq 0.5$), according to the heat flux $Q_\bot$ at edge $r=0$, and then decreases exponentially in the scrape-off layer ($0.5\leq r\leq 1$).  These numerical results correspond to the retarding field  analyzer (RFA)~\cite{Kocan200879,Kocan2008,Kocan2009}.

 \begin{center}
\begin{figure}[htbp]
 \begin{tabular}{cc}
\includegraphics[width=7.5cm]{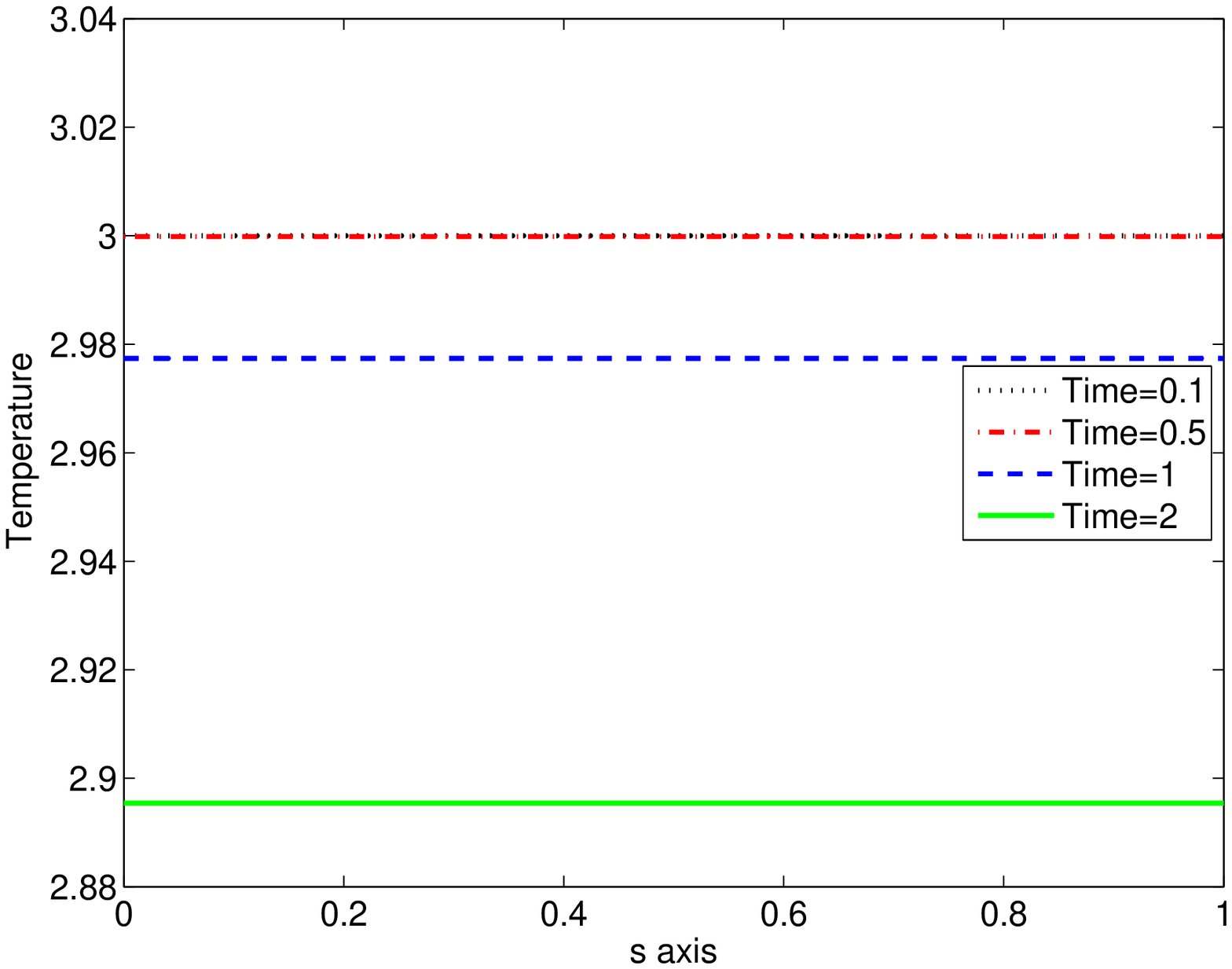} & 
\includegraphics[width=7.5cm]{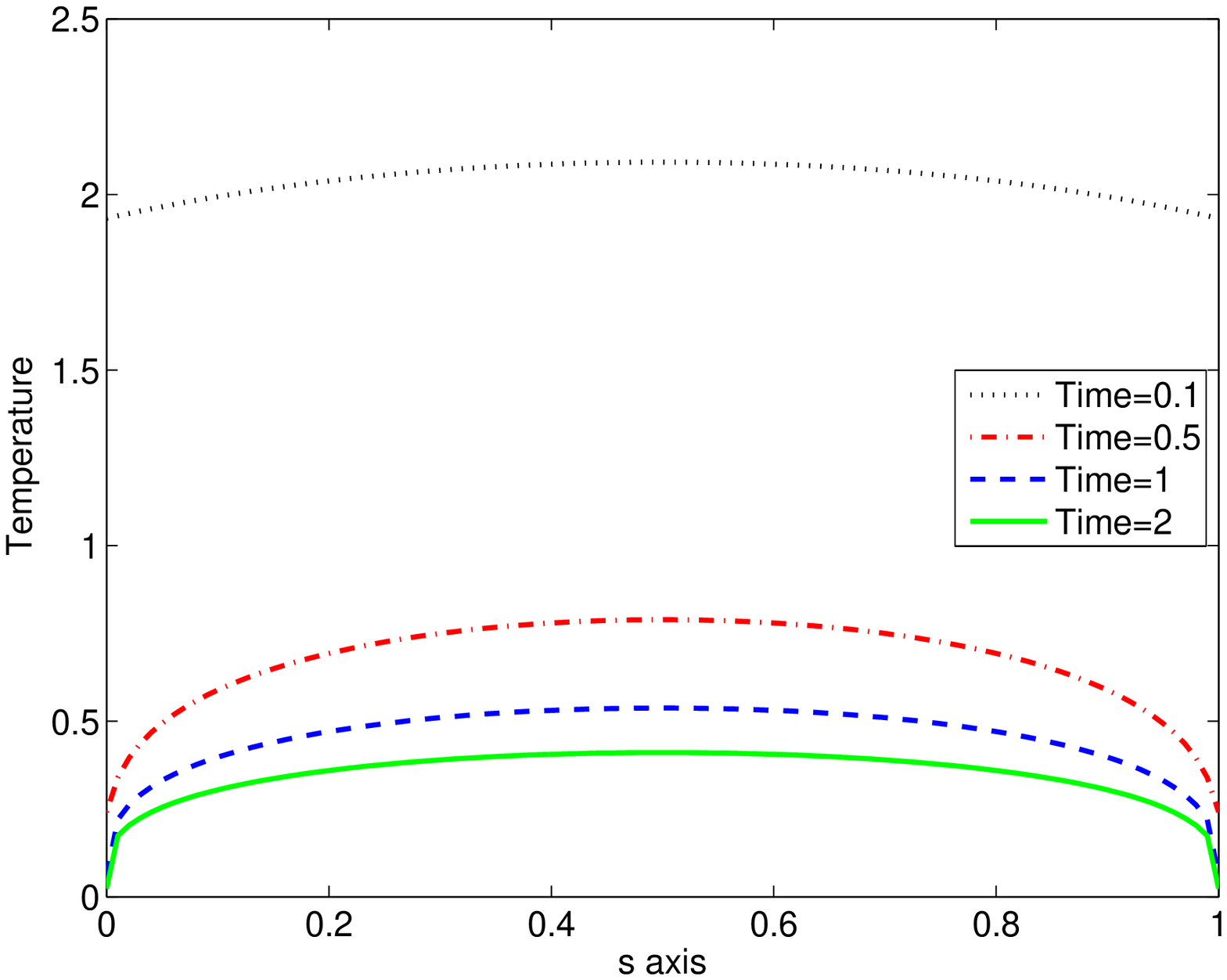}
\\
(a) $r=1/4$ &
(b) $r=3/4$ 
\\
      \includegraphics[width=7.5cm]{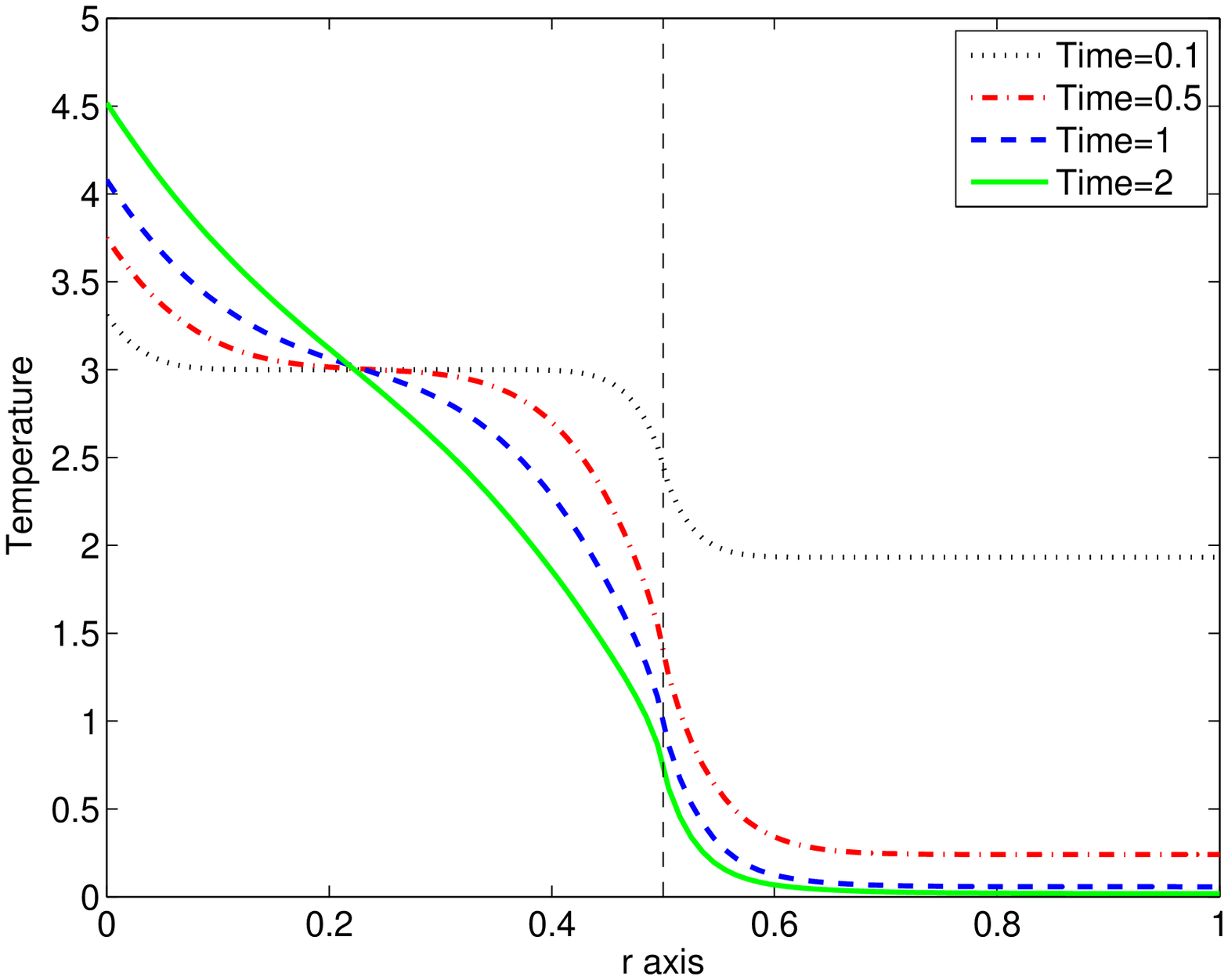} &
      \includegraphics[width=7.5cm]{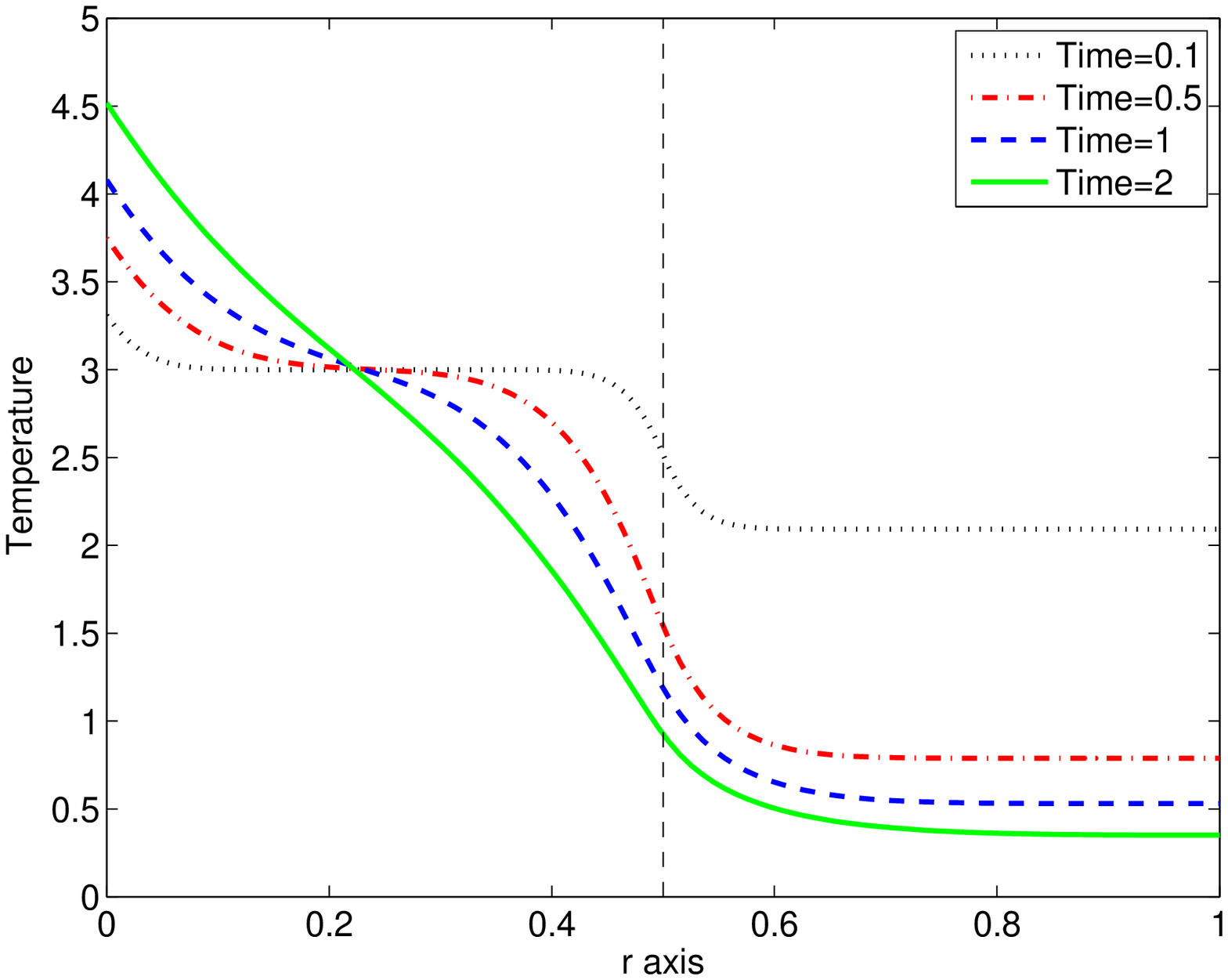}
\\
(c) $s=10^{-2}$ & 
(d) $s=1/2$
 \end{tabular}   
\caption{Temperature evolution at section  $r=1/4$, $r=3/4$, $s=10^{-2}$ and $s=1/2$  at time $t=0.1$, $0.5$, $1$ and $2$ respectively.}
  \label{fig2Dcoupes}
 \end{figure}
 \end{center}

Finally, we present the evolution of the energy dissipation with respect to time:
\begin{eqnarray*}
\frac{1}{2} \frac{d}{dt}\int_{\Omega} |T(t,s,r)|^2dsdr &=& \E_1 \,+\,\E_2 \,+\,\E_3,
\end{eqnarray*}
with 
$$
\left\{\begin{array}{l}
\displaystyle \E_{1}\,:=\,- \int_\Omega \left( K_\parallel\left(T\right)^{5/2}\left|\partial_sT\right|^2 \,+\, K_\perp\left|\partial_r T\right|^2\,\right)\,dr\,ds,
\\
\,
\\
\displaystyle \E_{2}\,:=\, -\gamma\,\int^1_{1/2}T(t,1,r)^2  \,+\, T(t,0,r)^2 \,dr,
\\
\,
\\
\displaystyle \E_3 \,:=\, +Q_\bot\, K_\perp\,\int^1_0T(s,0)ds.
\end{array}\right.
$$ 
The Figure \ref{figeq2Denergy} states the terms $\E_{1}$, $\E_{2}$, $\E_{3}$ as function of $t$. We plot these terms obtained from  implicit and IMEX schemes. Note that these two figures are almost the same. In fact, at the beginning of simulation, there is a fast decay of the temperature, thus  the quantity $-\E_{1}$ representing the total energy exchange ratio in the domain $\Omega$,  is increasing for $t<0.1$. Then, it converges to an equilibrium state for larger time. On the other hand, the quantity $-\E_{2}$ decreases with respect to time, it is due to the anisotropy between $K_\parallel$ and $K_\bot$. Indeed, the energy is transferred to the limiters in the  scrape-off layer region whereas in the perpendicular direction $r$, the thermal diffusivity is  small. Finally, as we have seen in Figure~\ref{fig2Dcoupes} on  the edge of  of the core, the temperature does not vary significantly, thus the quantity $\E_{3}$ increases slightly with respect to time.

\begin{center}
\begin{figure}[htbp]
  \begin{tabular}{cc}
\includegraphics[width=7.5cm]{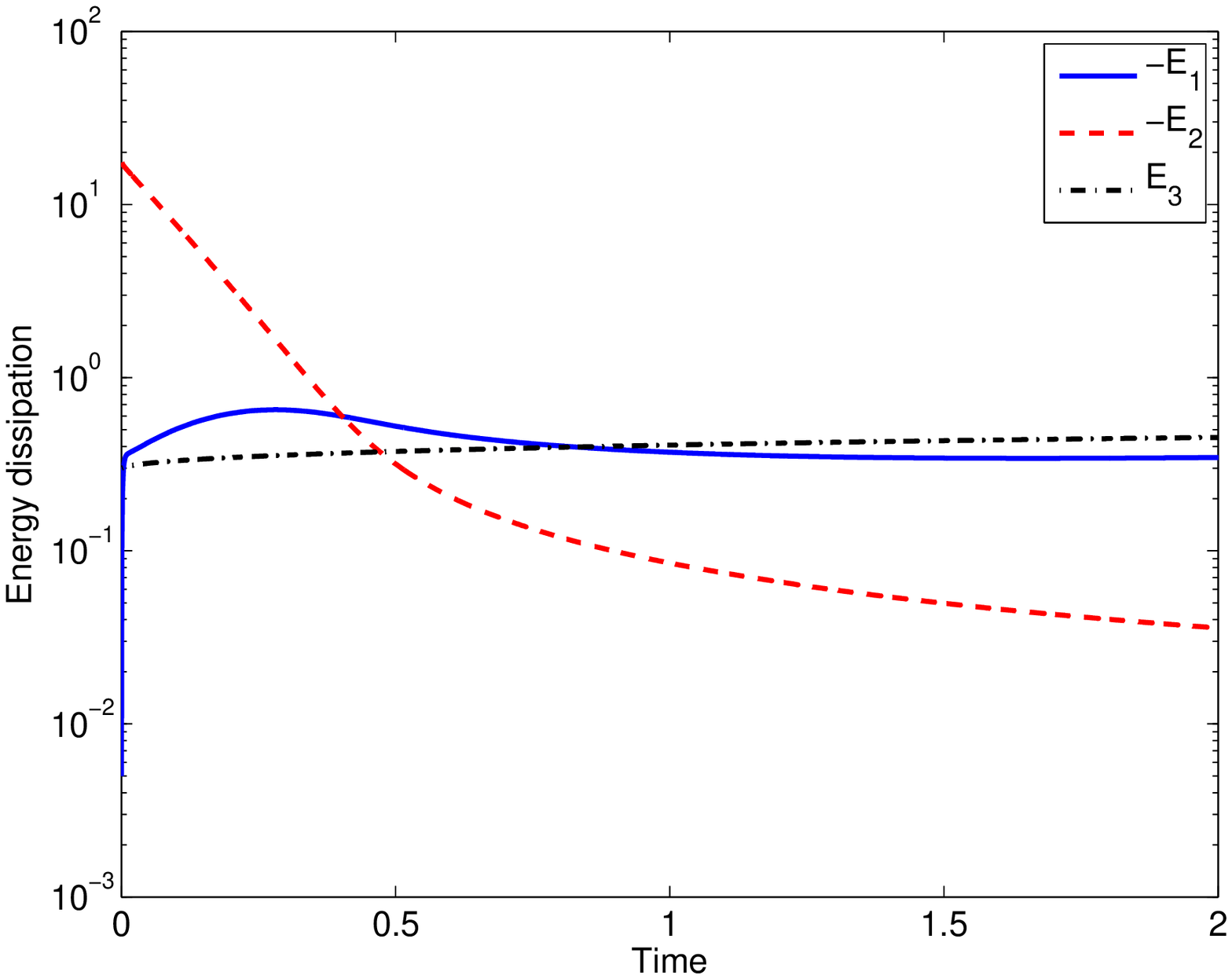} &
\includegraphics[width=7.5cm]{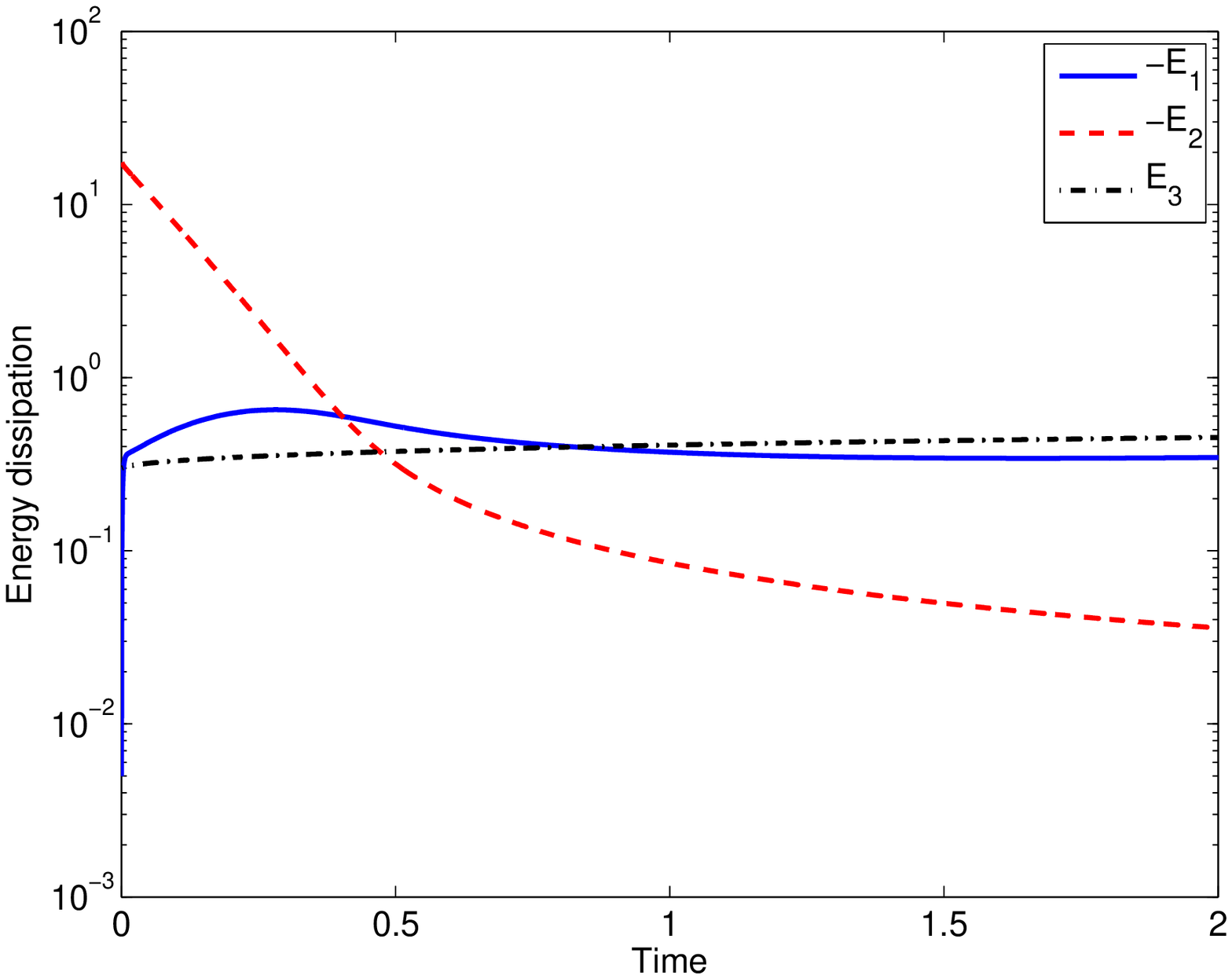}
\\
(a) Implicit scheme & (b) IMEX scheme 
\end{tabular}
\caption{Evolution of the energy dissipation with respect to time for problem (\ref{eq2D}), with $\Delta t=0.001$.}
 \label{figeq2Denergy}
 \end{figure}
 \end{center}


\section{\bf The coupling problem}
\label{sec:4}
\setcounter{equation}{0}

In this section, we consider the full 2D model (\ref{temp_simpl}) composed of two different particle species, {\it i.e.} ions and electrons. We denote by $T_i$ ({\it resp.} $T_e$) the temperature of ions ({\it resp.} electrons) which depends  on time $t$ and two space variables $(s,r)\in\Omega$. The two equations are coupled by a non-zero source term which balances the temperature between the two particle species,
\begin{equation}
\left\{
\begin{array}{l}
 \displaystyle{ \partial_tT_i-\partial_s(K_{\parallel,i}T^{5/2}_i\partial_sT_i)-\partial_r(K_{\bot,i}\partial_rT_i)=+\beta(T_i-T_e),} \textrm{ for }(s,r)\in\Omega,
\\
\,
\\
\displaystyle{\partial_tT_e-\partial_s(K_{\parallel,e}T^{5/2}_e\partial_sT_e)-\partial_r(K_{\bot,e}\partial_rT_e)=-\beta(T_i-T_e),} \textrm{ for }(s,r)\in\Omega,
\end{array}\right.
\label{eqcoulage}
\end{equation}
where $K_{\bot,i}\ll K_{\parallel,i}$, $K_{\bot,e}\ll K_{\parallel,e}$ and $\beta$ is a negative constant. These two equations are completed with the same type of boundary conditions as in \eqref{eq2Dbcr}-\eqref{eq2Dbcs1}. 


\subsection{Time splitting scheme}
Now we discretize the full system (\ref{eqcoulage}) using a splitting scheme in three steps. We assume that an approximation of the solution $(T_e,T_i)$ at time $t^n$ is known and denote it by $(T_e^n,T_i^n)$. Therefore, we first approximate the source part coupling the two temperatures $T_ e$ and $T_i$ using an  implicit scheme, which yields 
\begin{equation}
  \left\{
  \begin{array}{l}
\displaystyle  T_i^{\star}\,=\,\frac{1}{2}\left(1\,-\,\frac{1}{1-2\beta\Delta t}\right)\,T_e^{n}\,+\,\frac{1}{2}\,\left(1\,+\,\frac{1}{1-2\beta\Delta t}\right)\,T_i^{n},
\\
\,
\\
\displaystyle T_e^{\star}=\frac{1}{2}(1+\frac{1}{1-2\beta\Delta t})T_e^{n}+\frac{1}{2}(1-\frac{1}{1-2\beta\Delta t})T_i^{n}.
  \end{array}
\right.
\label{eqschemeequilibrium}
\end{equation}
It is clear that \eqref{eqschemeequilibrium} guarantees the positivity of the temperature. Then we apply the same time splitting steps as before in direction $s$ and in direction $r$ as follows. On the one hand we compute $T_\alpha^{\star\star}$ for $\alpha\in\{i,\,e\}$ by solving (\ref{sch:101})-(\ref{sch:103}). On the other hand we apply the last step (\ref{sch:11})-(\ref{sch:12}) in the direction $r$.

Furthermore, for the scheme (\ref{sch:101})-(\ref{sch:103}), (\ref{sch:11})-(\ref{sch:12}) and (\ref{eqschemeequilibrium}), we also  prove an energy estimate 
\begin{proposition}
Consider that the initial datum $T_0$ is nonnegative and $T_0\in L^\infty(0,1)$. Assume that for $\alpha\in\{i,\,e\}$, the viscosity term $\nu$ is such that  for any $r\in(0,1)$,
\begin{equation*}
\max_{\alpha\in\{i,e\}}K_{\parallel,\alpha}\|T^n_\alpha\|^{5/2}_\infty\,\leq\,\nu,\quad \forall n\in\NN.
\end{equation*}
Then the numerical solution, given by (\ref{eqschemeequilibrium}), satisfies the following 
\begin{eqnarray*}
&&\frac{1}{2}\sum_{\alpha\in\{i,e\}} \int_{\Omega}\left[\,|T^{n+1}_\alpha|^2 \;+\, \Delta t\,\nu\,|\partial_sT^{n+1}_\alpha|^2\,\right] dr\,ds
\\
&\leq& \frac{1}{2}\sum_{\alpha\in\{i,e\}}\int_{\Omega}\left[\,|T^{0}_\alpha|^2 \,+\,\Delta t\,\nu\,|\partial_sT^{0}_\alpha|^2\,\right]dr\,ds
\\
&-& \Delta t\,\sum_{\alpha\in\{i,e\}} \sum_{k=1}^{n+1}K_{\perp,\alpha}\,\int_{\Omega}|\partial_rT^{k}_\alpha|^2\,dr\,ds
\\
&+& \Delta t\,\sum_{\alpha\in\{i,e\}} \sum_{k=1}^{n+1}K_{\perp,\alpha}\, Q_{\perp,\alpha}\,\int_{0}^1 T^{k}_\alpha(s,0) ds.
\end{eqnarray*}
\end{proposition}
\begin{proof}
We first observe that the energy estimate of the two last steps in the direction $s$ and $r$ are the same as the one proved in Proposition~\ref{prop:2D}, hence we have 
\begin{eqnarray*}
\frac{1}{2}\int_{\Omega}\left[ \,|T^{n+1}_\alpha|^2 \,+\, \nu\,\Delta t\,\left(\,|\partial_sT^{n+1}_\alpha|^2 \,+\,K_{\perp,\alpha}\,|\partial_rT^{n+1}_\alpha|^2 \right)\,\right]\,dr\,ds 
\\
\leq\, \frac{1}{2}\int_{\Omega}\left[\,|T^{\star\star}_\alpha|^2\, +\, \Delta t\,\nu\,|\partial_sT^{\star\star}_\alpha|^2\,\right]dr\,ds \,+\, \Delta t\,K_{\perp,\alpha}\,Q_{\perp,\alpha}\,\int_{0}^1 T^{n+1}_\alpha(s,0)\,ds. 
\end{eqnarray*}
Therefore, to achieve the proof on the energy estimate,  we only observe that  \eqref{eqschemeequilibrium} can be written as follows 
\begin{equation}
\left\{
\begin{array}{l}
\displaystyle T_i^{\star}\,-\,T^{n}_i \,=\, +\,\Delta t\,\beta\,\left(T^{\star}_i\,-\,T^{\star}_e\right),
\\
\,
\\
\displaystyle  T_e^{\star}\,-\,T^{n}_e\,=\, -\,\Delta t\,\beta\,\left(T^{\star}_i\,-\,T^{\star}_e\right). 
\end{array}\right.
\label{pierre}
\end{equation}
Multiplying the first  equation (\ref{pierre}) by $T_i^\star$ and the second by $T_e^\star$ and  integrating on $(r,s)\in\Omega$, it yields
\begin{equation*}
  \frac{1}{2}\int_{\Omega}|T_i^{\star}|^2\,+\,|T_e^{\star}|^2drds \,\leq\, \frac{1}{2}\int_{\Omega}|T_i^{n}|^2 \,+\,|T_e^{n}|^2drds
\end{equation*}
Moreover, differentiating  \eqref{pierre}  with respect to $s$ and multiplying the first equation by $\nu\partial_sT_i^{\star}$ and the second one by $\nu\partial_sT_e^{\star}$, we get
\begin{equation*}
  \frac{\nu}{2}\int_{\Omega}|\partial_s T_i^{\star}|^2\,+\,|\partial_s T_e^{\star}|^2drds \,\leq\, \frac{\nu}{2}\int_{\Omega}|\partial_s T_i^{n}|^2 \,+\,|\partial_s T_e^{n}|^2drds.
\end{equation*}
Finally, we have
\begin{eqnarray*}
 && \frac{1}{2}\sum_{\alpha\in\{i,e\}}\int_{\Omega}\left[\,|T^{n+1}_\alpha|^2\,+\Delta t\,\left(\nu\,|\partial_sT^{n+1}_\alpha|^2 +K_{\perp,\alpha}\,|\partial_rT^{n+1}_\alpha|^2 \right)\,\right]\,dr\,ds\,
\\
&\leq& \frac{1}{2}\sum_{\alpha\in\{i,e\}}\int_{\Omega}\left[\,|T^{n}_\alpha|^2\,+\Delta t\,\left(\nu\,|\partial_sT^{n}_\alpha|^2 +K_{\perp,\alpha}\,|\partial_rT^{n}_\alpha|^2 \right)\,\right]\,dr\,ds\,
\\
&-& \Delta t\sum_{\alpha\in\{i,e\}}\,K_{\perp,\alpha}\,Q_{\perp,\alpha}\,\int_{0}^1 T^{n+1}_\alpha(s,0) ds.
\end{eqnarray*}
Summing over $k=0,\ldots,n$, we complete the proof.
\end{proof}

Finally space discretization is performed using the finite volume scheme presented in Section~\ref{sec:discretspace2D}.

\subsection{Numerical results}
In this section, we compare  the numerical results obtained from the implicit scheme and the IMEX scheme for \eqref{eqcoulage}. We choose $K_{\parallel,i}=2\times0.01$, $K_{\parallel,e}=1$, $K_{\bot,i}=0.01$, $K_{\bot,e}=0.01$, $\gamma_i=0$, $\gamma_e=2.5$, $Q_{\bot,i}=Q_{\bot,e}=10$ and $\beta=-0.02$. The initial temperature is such that
$$
T_i^0(s,r) \,=\, 3,\textrm{ and  }T_e^0(s,r)\,=\,3, \,(s,r)\in\Omega.
$$

The final time of the simulation is $T_{end}=1$ and the mesh size is chosen as  $n_s=100$, $n_r=100$.

We plot the electron and ion temperature and compare their ratio at different time. The aim is to compare the different behaviors between electron and ion temperatures at the edges and in the scrape-off layer of a Tokamak~\cite{Isoardi2011}. 

On the one hand, we propose in Figure~\ref{figcouplage}, the temperature evolution. On the left hand side, we present the electron temperature, whereas on the right hand side we give the ion temperature.  We first notice that the electron parallel thermal diffusivity is about $100$ times larger than the one for ions~\cite{bragi, wesson}, and the electron energy exchange ratio at the edge $r\in(0.5,1)$ depends on $O(T_e^{-3/2})$, thus the temperature has a fast decay when it is small in the scrape-off layer. However, the boundary conditions for ions in the scrape-off layer is given by the homogeneous Neumann condition $\partial_s T_i=0$, which means that  there is no energy exchange at the limiters. Thus the ion temperature does not vary significantly at scrape-off layer.

\begin{center}
\begin{figure}[htbp]
\begin{tabular}{cc}   
\includegraphics[width=7.5cm]{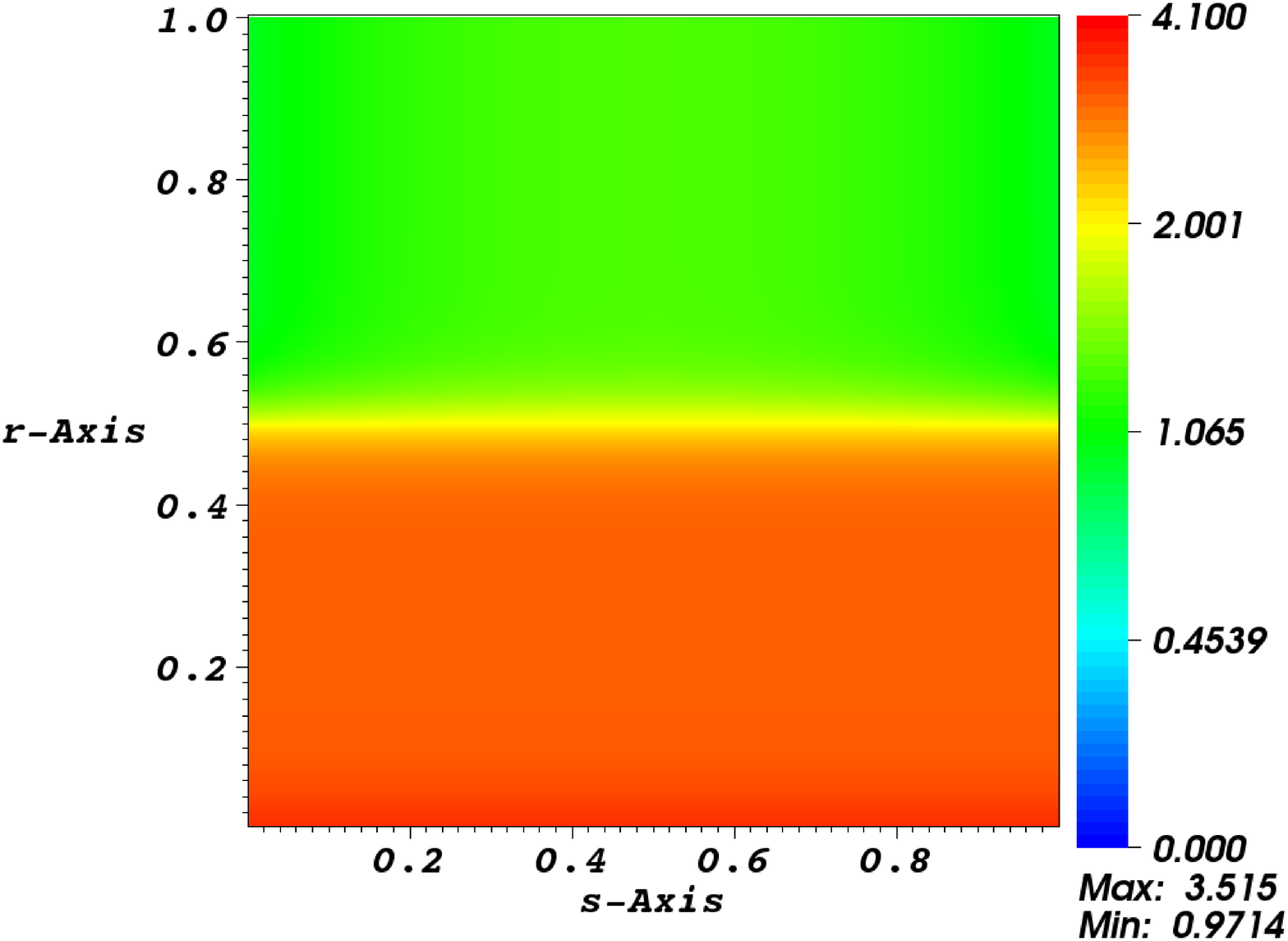}&
\includegraphics[width=7.5cm]{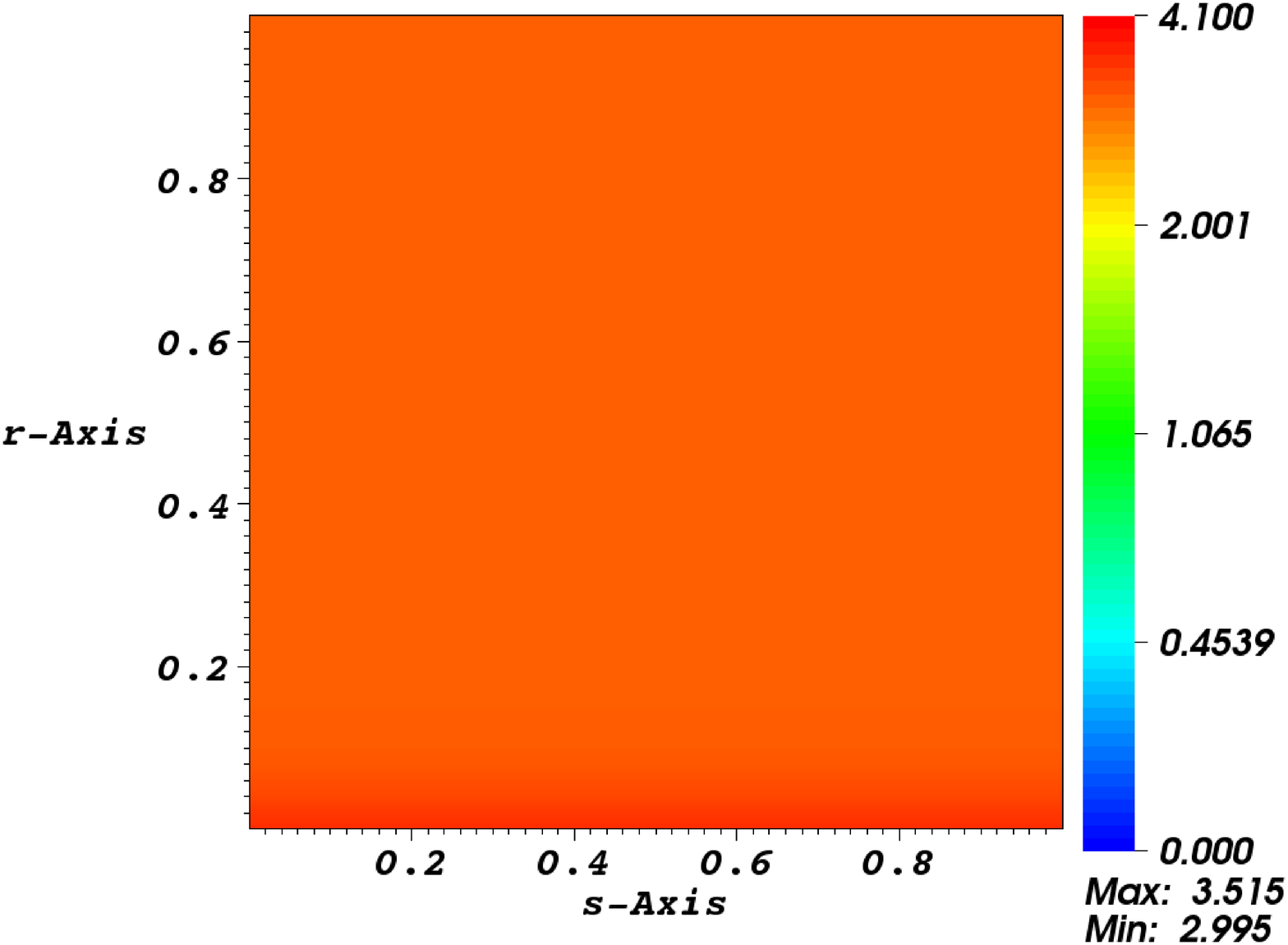}
\\
(a) $T_e$ at $t=0.25$ &
(b) $T_i$ at $t=0.25$
\\
\includegraphics[width=7.5cm]{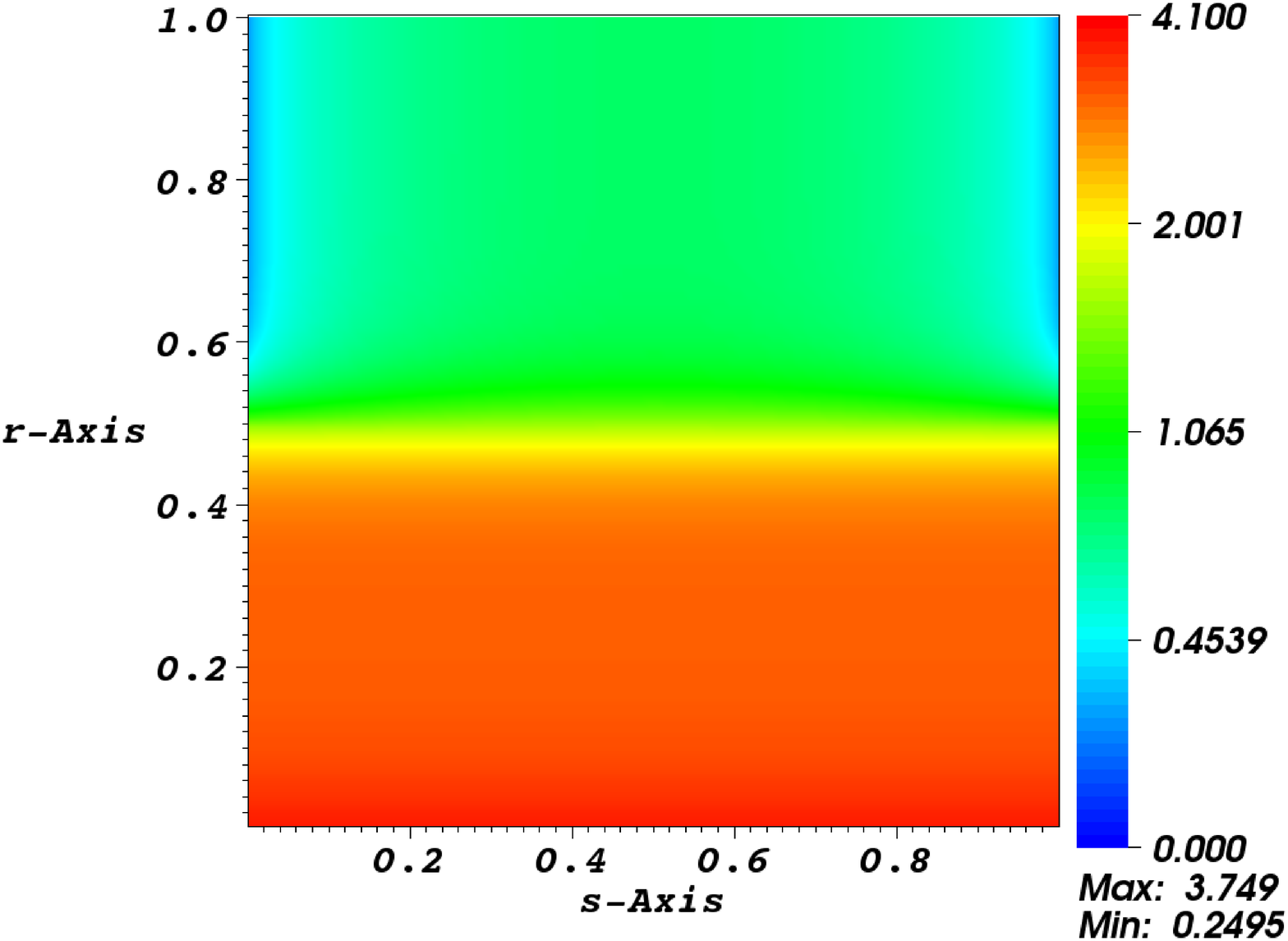}&
\includegraphics[width=7.5cm]{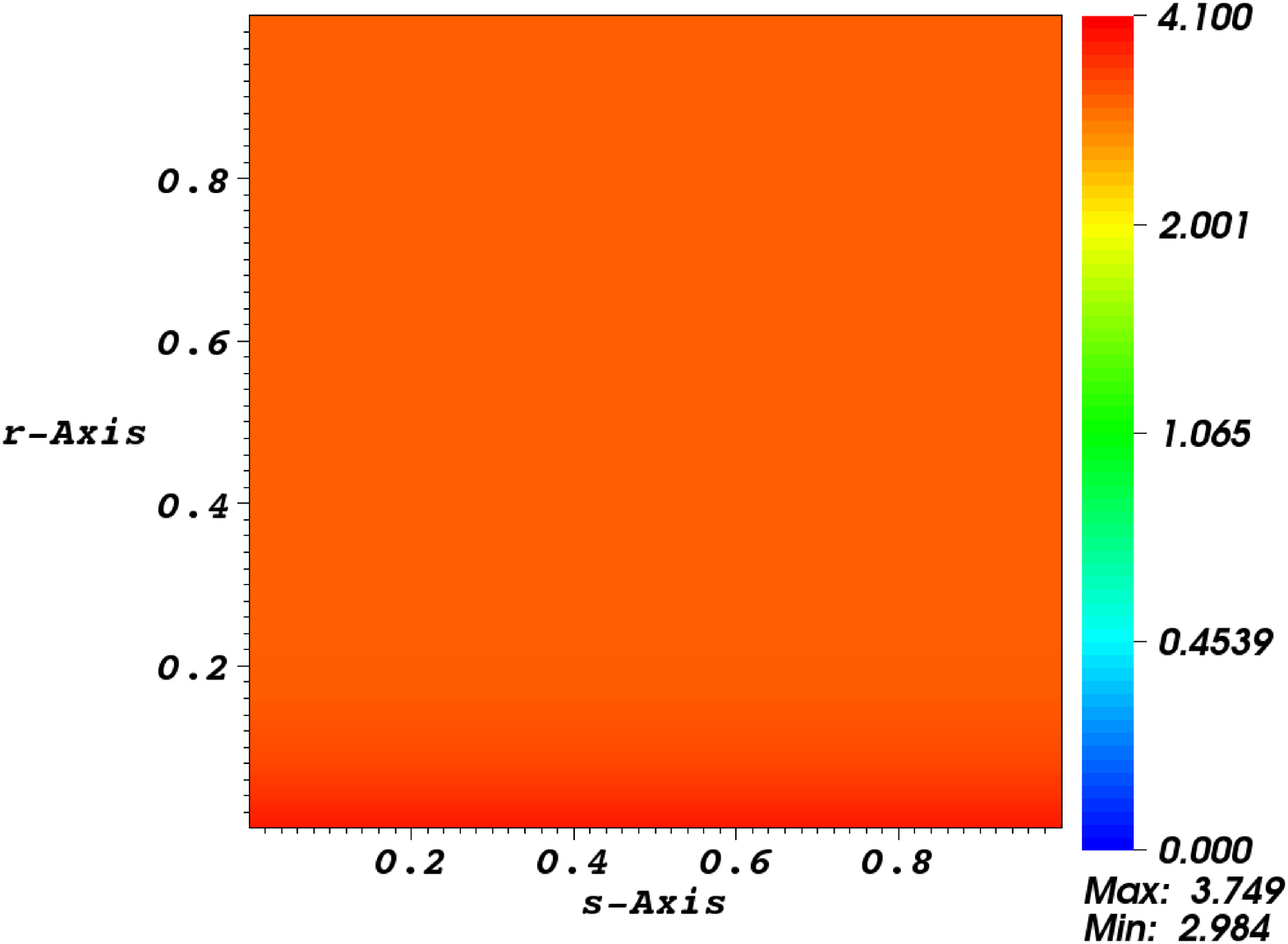}
\\
(c) $T_e$ at $t=0.5$ &
(d) $T_i$ at $t=0.5$
\\      
\includegraphics[width=7.5cm]{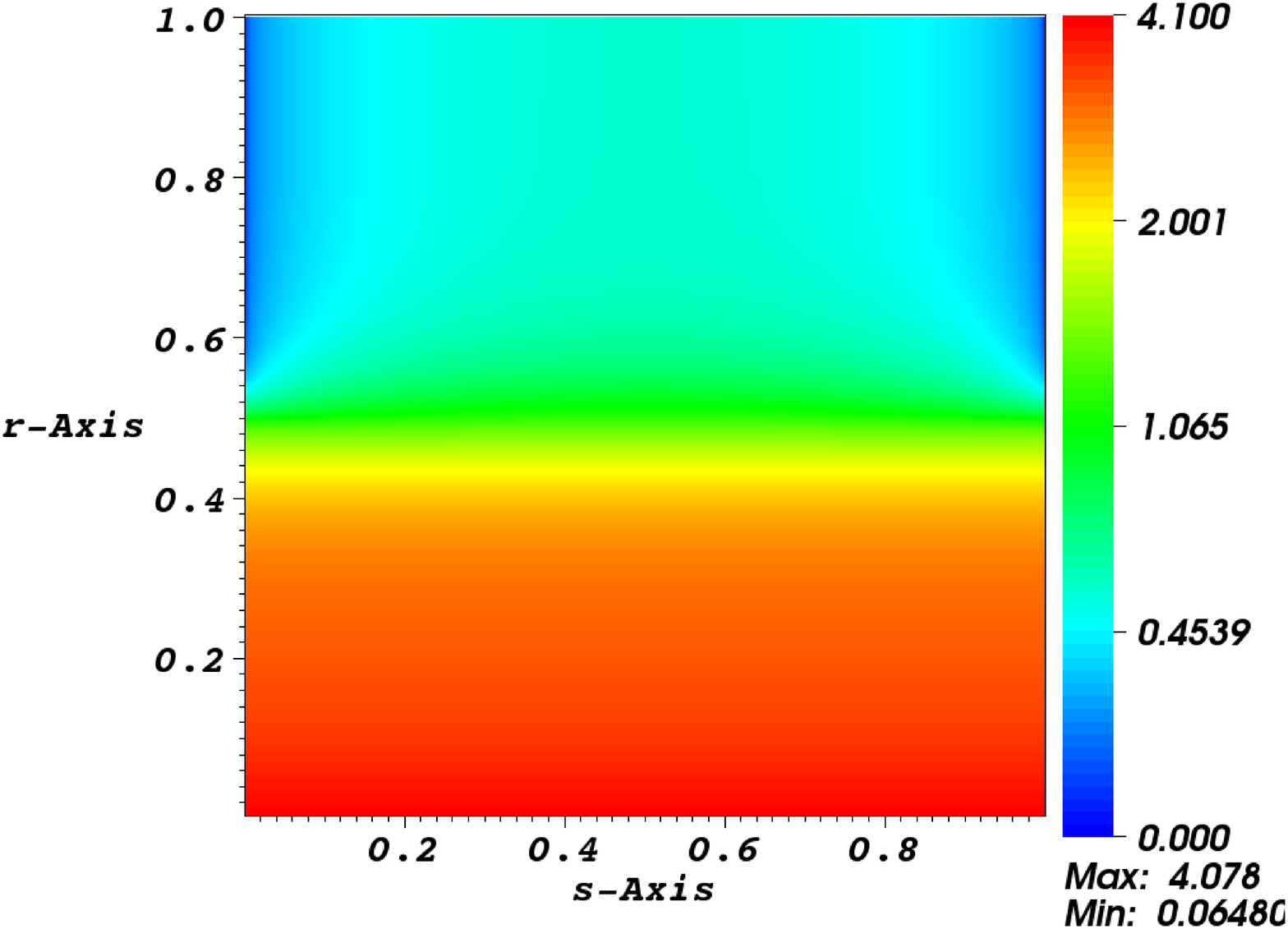}&
\includegraphics[width=7.5cm]{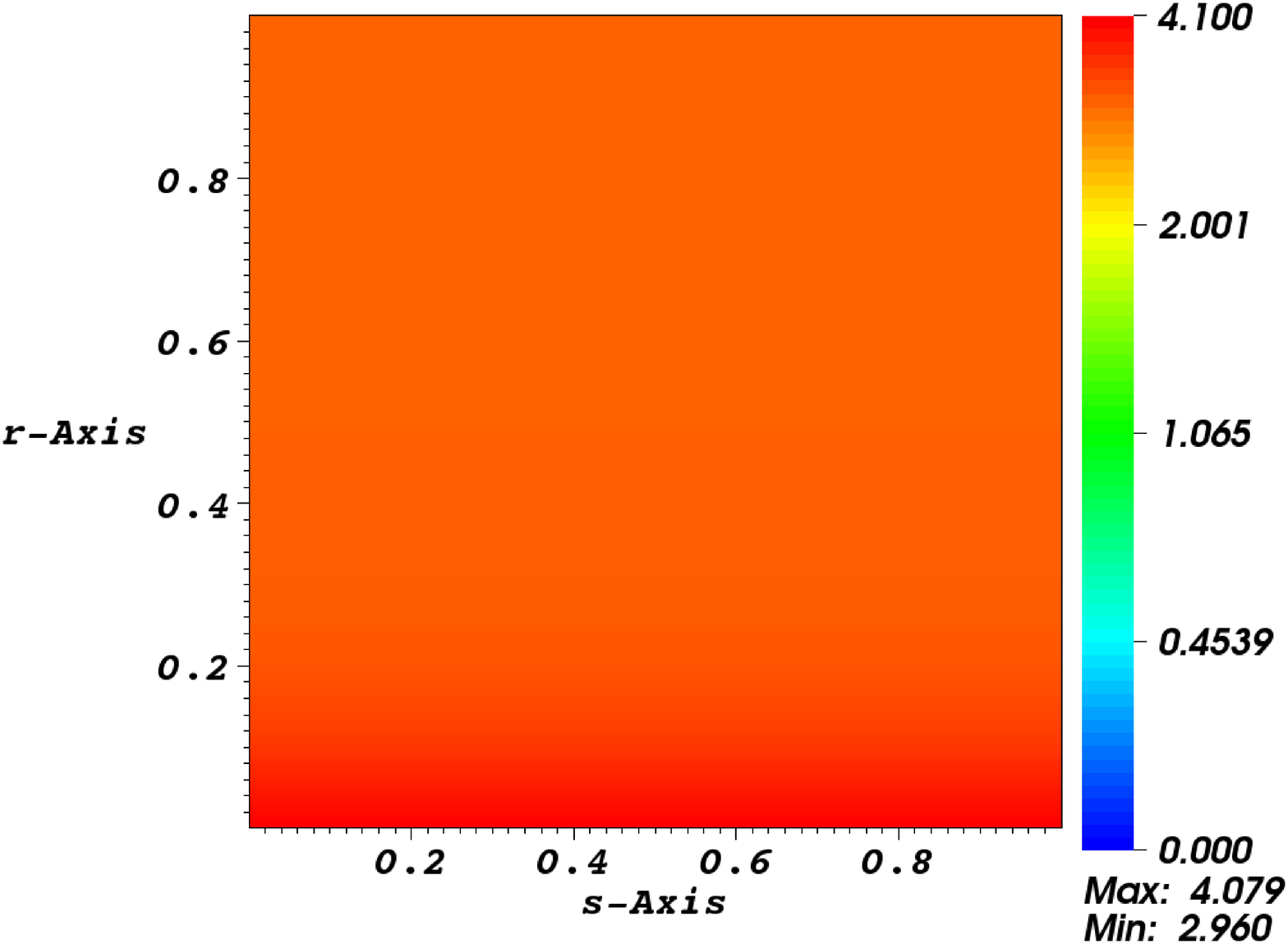}
\\
\
(d) $T_e$ at $t=1$ &
(e) $T_i$ at $t=1$
\end{tabular}
\caption{Temperature evolution  of problem (\ref{eqcoulage}).}
\label{figcouplage}
\end{figure}  
\end{center}

On the other hand,  the ratio between electron temperature and ion temperature is presented in Figure~\ref{fig2Dcoupes_couplage}. The Figure~\ref{fig2Dcoupes_couplage} illustrates that in the transition layer, the ion and electron temperatures are almost identical. However, in the scrape-off layer, at the final time $T_{end}=1$ the ratio $\tau$ becomes large around the limiters due to the boundary condition $\partial_s T_e\varpropto T_e^{-3/2}$. The evolution of the ratio $\tau$ in the radial direction is given in Figures~\ref{fig2Dcoupes_couplage}. We observe that in the transition layer the ratio $\tau$ is almost equal to 1, whereas in the scrape-off layer this ratio becomes large. For example, at time $t=1$ the ratio $\tau=6$ for $s=1/2$ while it is $\tau=45$ for $s=10^{-2}$. These behaviors correspond to the experiment results in Ko\v can {\it et al.}~\cite{Kocan2008, Kocan2009}. At last we vary the parameter $\beta$ to study the equilibrium source term in Figure~\ref{fig2Dcoupes_equilibrium} and observe that when the parameter $|\beta|$ is large,  the ratio $\tau$ decreases.
 \begin{center}
\begin{figure}[htbp]
 \begin{tabular}{cc}
\includegraphics[width=7.5cm]{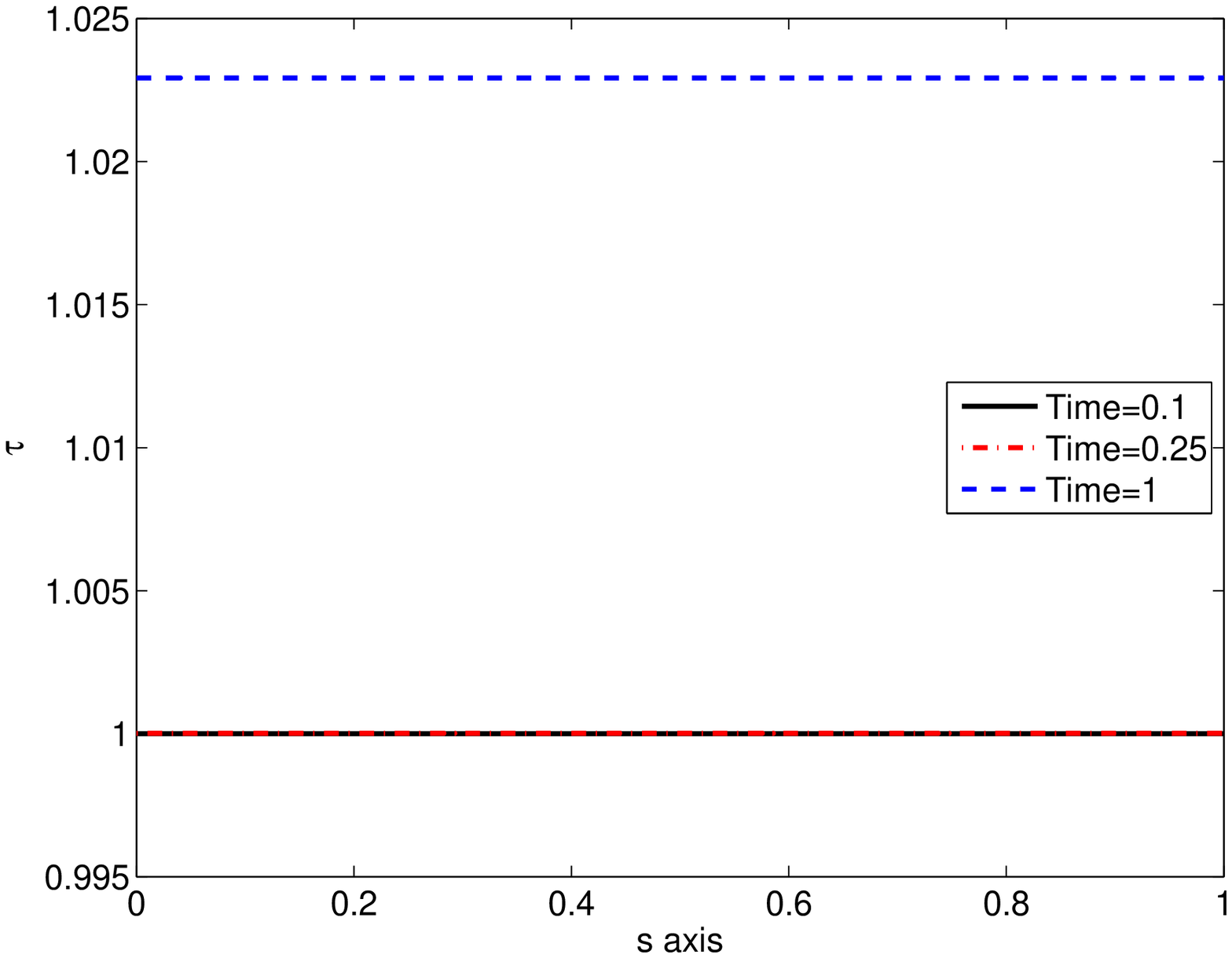}&
\includegraphics[width=7.5cm]{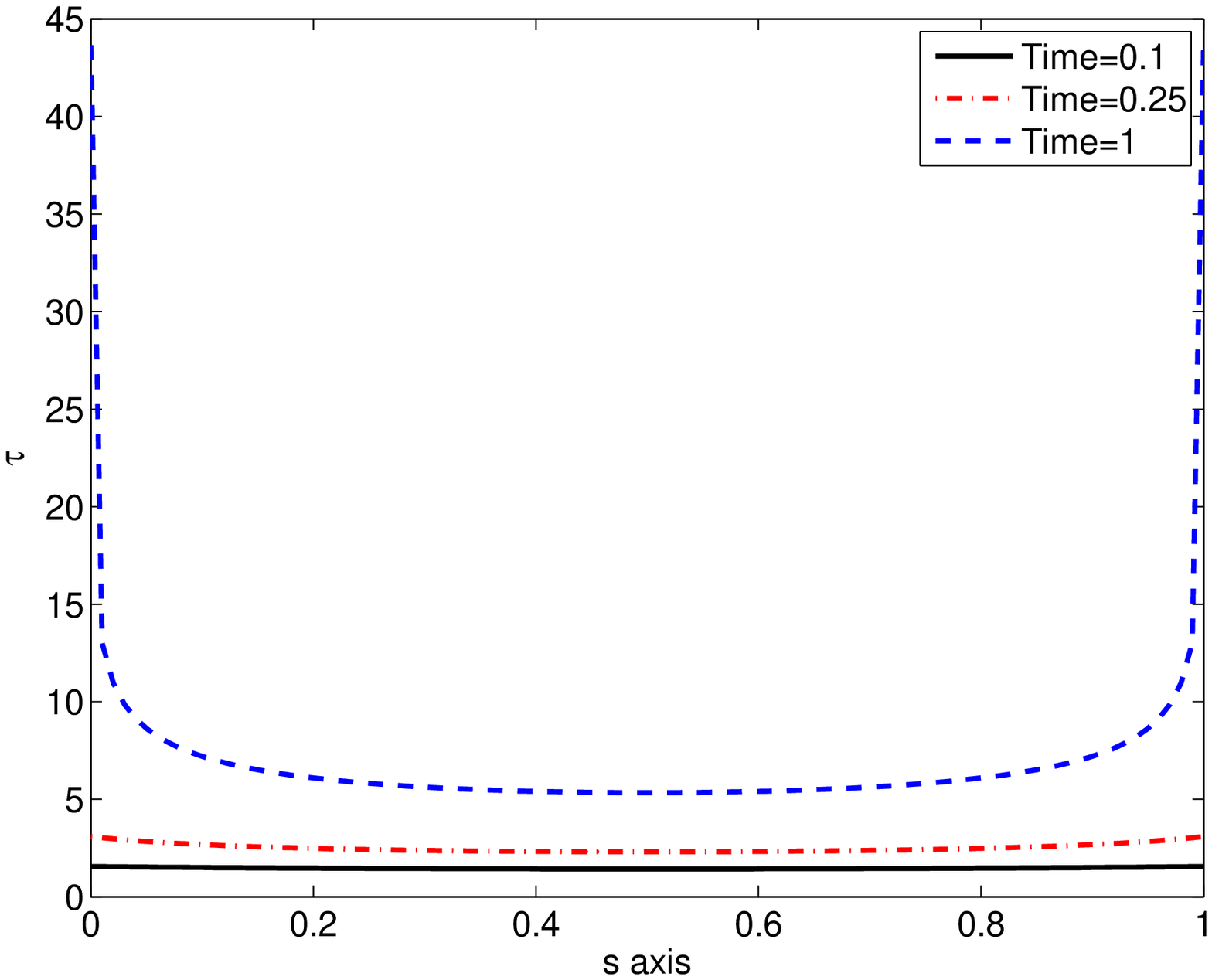}
\\
(a) $r=1/4$ &
(b) $r=3/4$   
\\
\includegraphics[width=7.5cm]{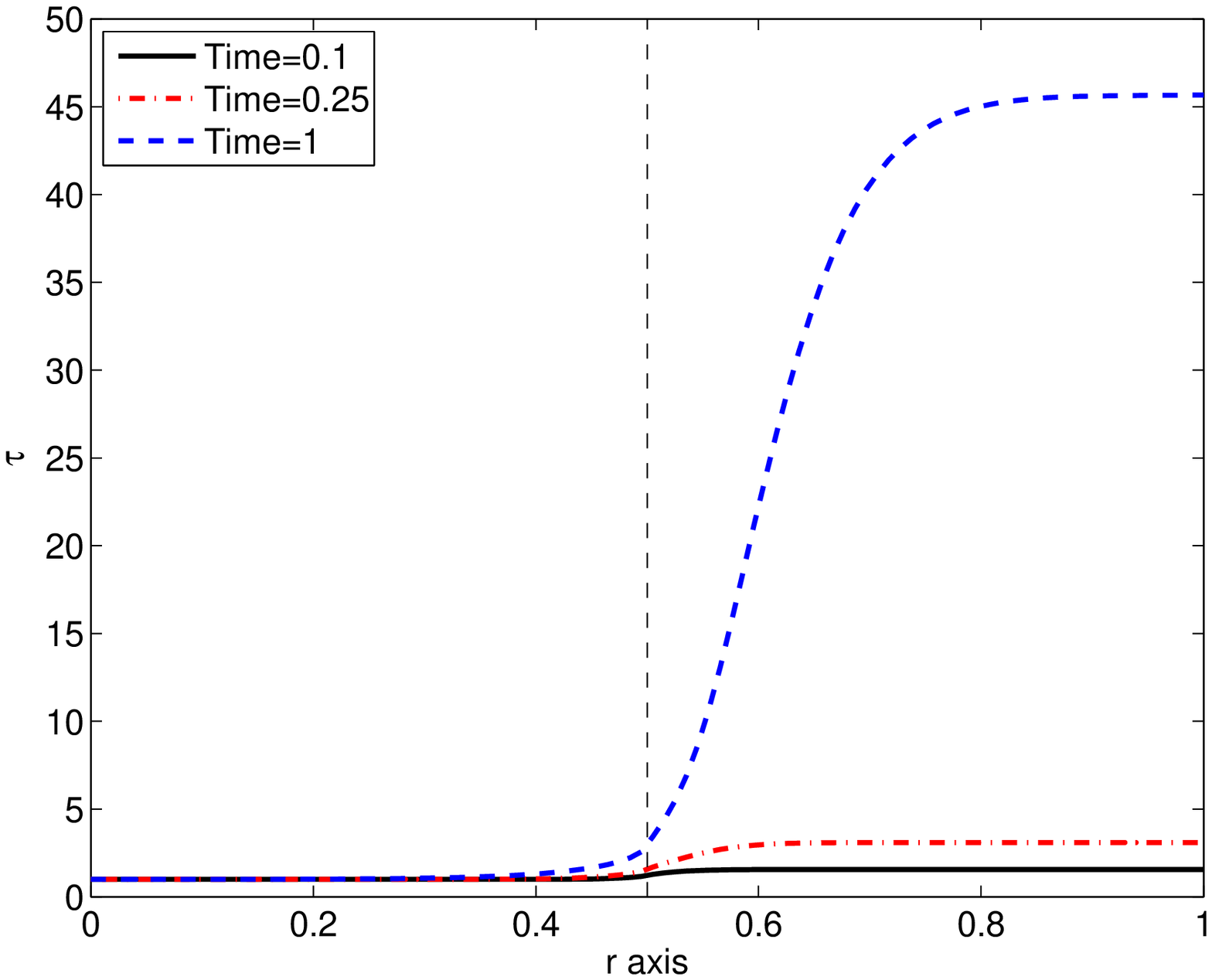} &
\includegraphics[width=7.5cm]{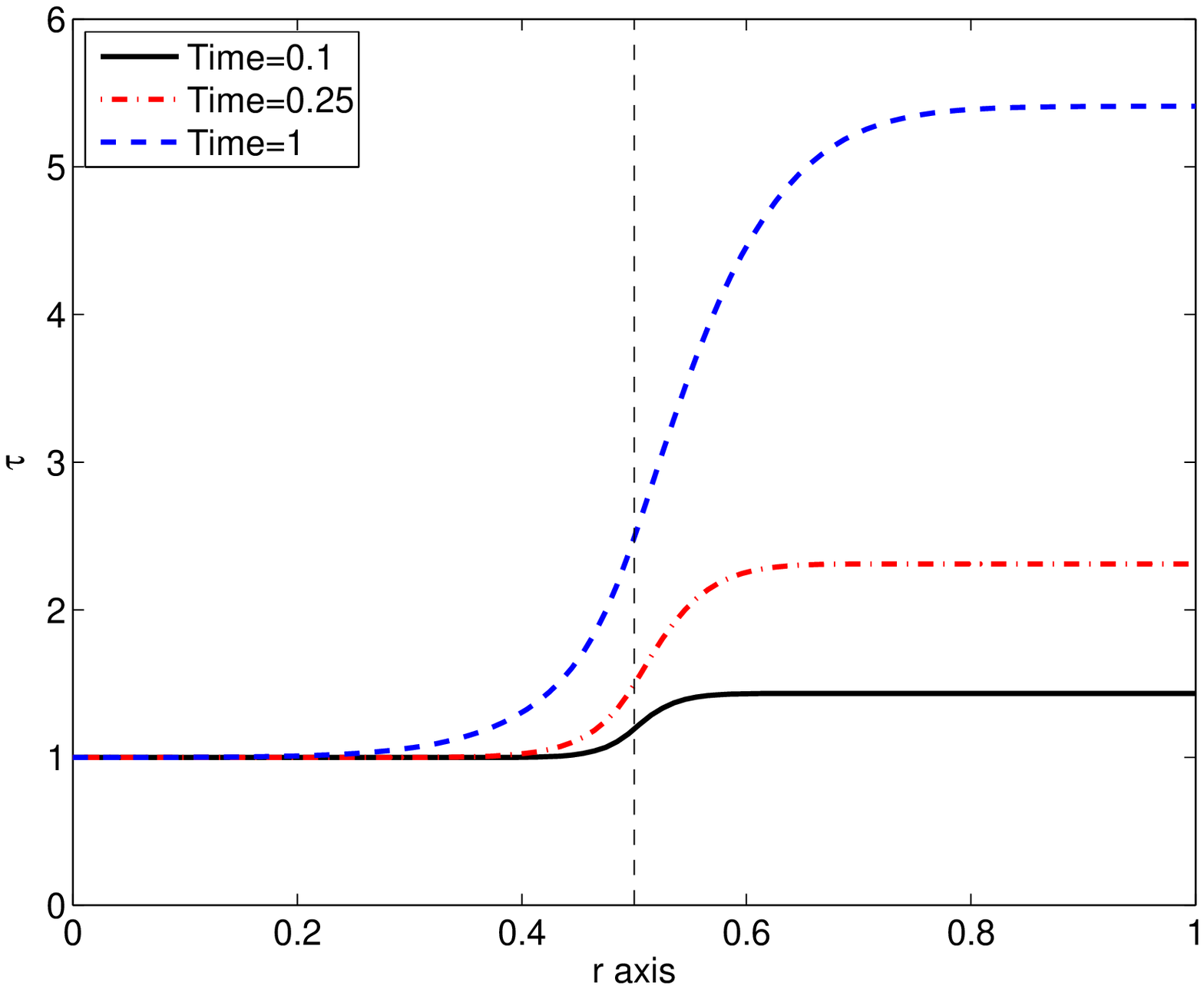}
\\
(c) $s=10^{-2}$ &
(d) $s=1/2$
\end{tabular}    
\caption{Ratio $\tau=T_i/T_e$ at section  $r=1/4$, $r=3/4$, $s=10^{-2}$ and $s=1/2$ at time $t=0.1$, $0.25$ and $1$ respectively.}
\label{fig2Dcoupes_couplage}
\end{figure}
\end{center}

\begin{center}
\begin{figure}[htbp]
\begin{tabular}{cc} 
\includegraphics[width=7.5cm]{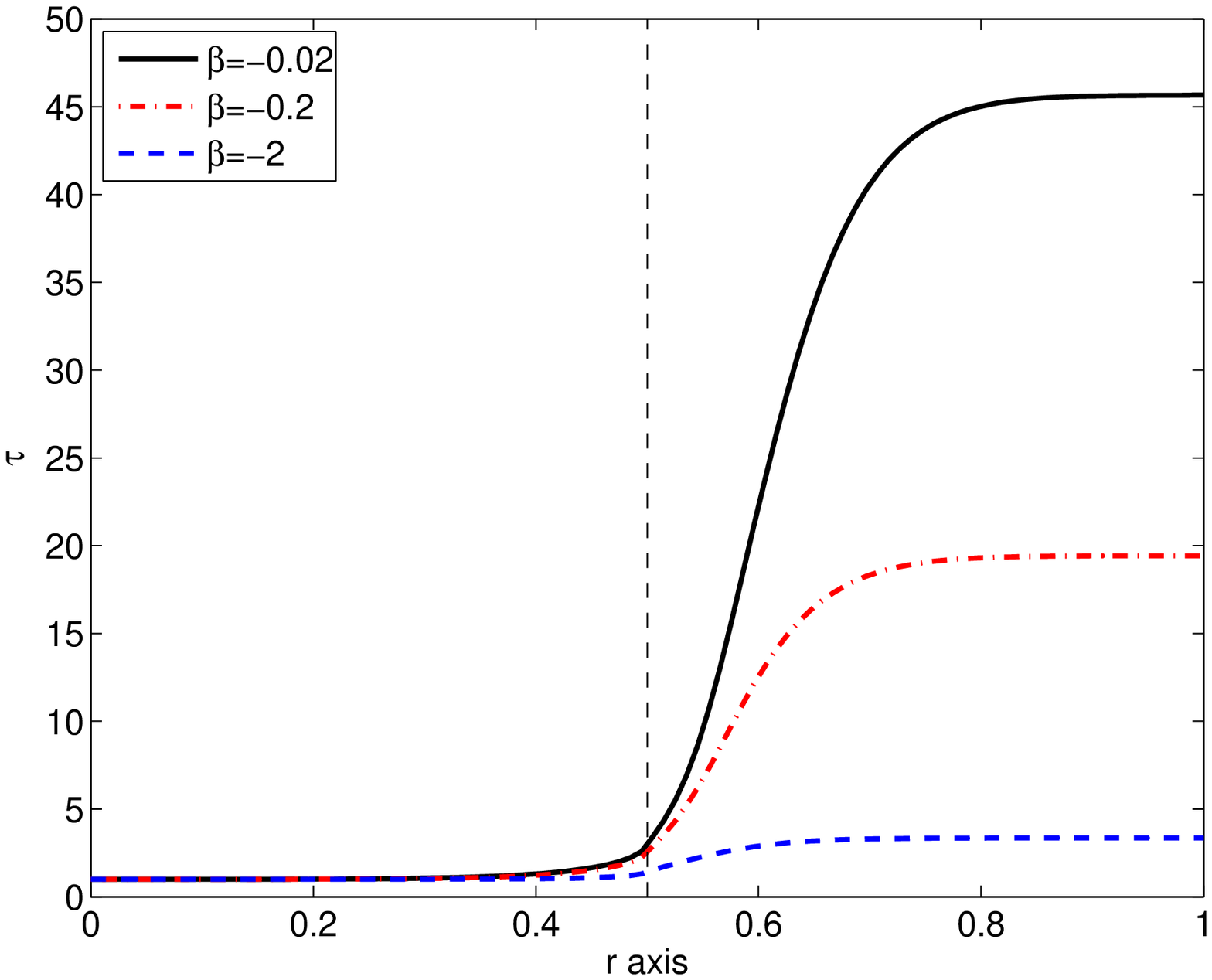}&
\includegraphics[width=7.5cm]{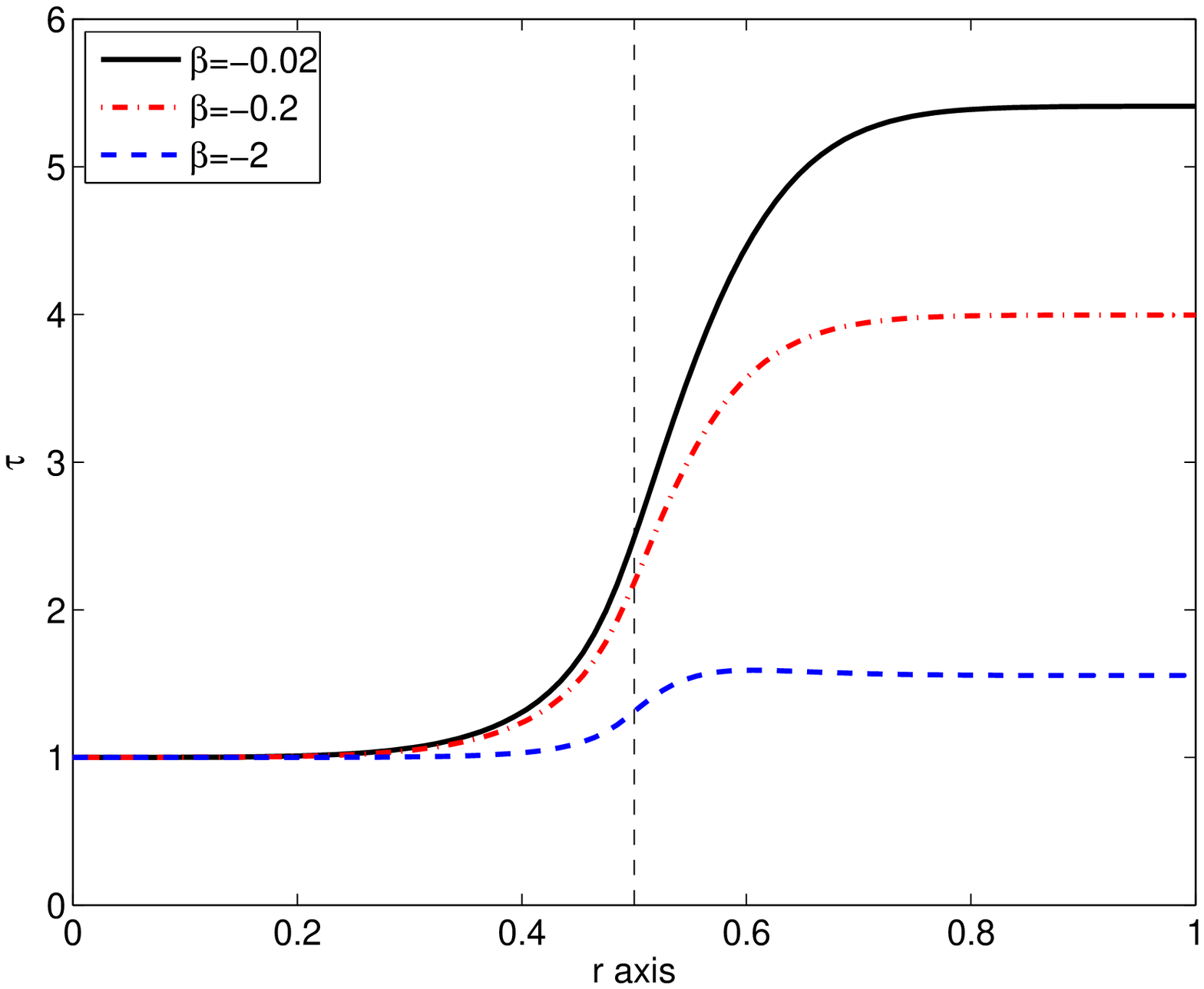}
\\
(a) $s=10^{-2}$ &
(b) $s=1/2$
\end{tabular}
\caption{Ratio $\tau=T_i/T_e$ at section $s=10^{-2}$ and $s=1/2$ for different parameters $\beta=-2\times10^{-2}$, $-2\times10^{-1}$ and $-2$ at time $t=1$.}
\label{fig2Dcoupes_equilibrium}
\end{figure}
\end{center}

\section{\bf Conclusion}
\label{sec:5}
\setcounter{equation}{0}

 We have presented various numerical approximations for a nonlinear temperature balance equation describing the heat evolution of a magnetically confined  plasma in the edge region of a tokamak. Numerical comparisons show that an IMEX scheme based on a ``smart'' decomposition of the nonlinear diffusive operator coupled with a splitting strategy gives an efficient numerical scheme in terms of accuracy, stability and reasonable computational cost. The next step would consists to couple the present model with the transport equations for the plasma density and momentum.


\bibliographystyle{plain}

\begin{flushleft} \signff \end{flushleft}
\vspace*{-44mm}
\begin{flushright} \signcn \end{flushright}
\begin{flushleft} \signcy \end{flushleft}

\end{document}

%% file: domaine.pstex_t
\begin{picture}(0,0)%
\special{psfile=domaine.pstex}%
\end{picture}%
\setlength{\unitlength}{1973sp}%
\begingroup\makeatletter\ifx\SetFigFont\undefined
\def\x#1#2#3#4#5#6#7\relax{\def\x{#1#2#3#4#5#6}}%
\expandafter\x\fmtname xxxxxx\relax \def\y{splain}%
\ifx\x\y   
\gdef\SetFigFont#1#2#3{%
  \ifnum #1<17\tiny\else \ifnum #1<20\small\else
  \ifnum #1<24\normalsize\else \ifnum #1<29\large\else
  \ifnum #1<34\Large\else \ifnum #1<41\LARGE\else
     \huge\fi\fi\fi\fi\fi\fi
  \csname #3\endcsname}%
\else
\gdef\SetFigFont#1#2#3{\begingroup
  \count@#1\relax \ifnum 25<\count@\count@25\fi
  \def\x{\endgroup\@setsize\SetFigFont{#2pt}}%
  \expandafter\x
    \csname \romannumeral\the\count@ pt\expandafter\endcsname
    \csname @\romannumeral\the\count@ pt\endcsname
  \csname #3\endcsname}%
\fi
\fi\endgroup
\begin{picture}(9699,6162)(1951,-7111)
\put(2626,-2536){\rotatebox{90.0}{\makebox(0,0)[lb]{\smash{{\SetFigFont{6}{7.2}{rm}radial direction $r$}}}}}
\put(6301,-2461){\makebox(0,0)[lb]{\smash{{\SetFigFont{6}{7.2}{rm}$\partial_r T = 0$}}}}
\put(10051,-6736){\makebox(0,0)[lb]{\smash{{\SetFigFont{6}{7.2}{rm}parallel direction $s$}}}}
\put(6451,-7036){\makebox(0,0)[lb]{\smash{{\SetFigFont{7}{8.4}{rm}CORE}}}}
\put(6076,-3736){\makebox(0,0)[lb]{\smash{{\SetFigFont{7}{8.4}{rm}Scrape off layer}}}}
\put(3301,-5986){\rotatebox{90.0}{\makebox(0,0)[lb]{\smash{{\SetFigFont{6}{7.2}{rm}periodic BC}}}}}
\put(10651,-5986){\rotatebox{90.0}{\makebox(0,0)[lb]{\smash{{\SetFigFont{6}{7.2}{rm}periodic BC}}}}}
\put(6076,-5536){\makebox(0,0)[lb]{\smash{{\SetFigFont{7}{8.4}{rm}Transition layer}}}}
\put(8326,-1786){\makebox(0,0)[lb]{\smash{{\SetFigFont{6}{7.2}{rm}Wall}}}}
\put(2626,-4186){\rotatebox{90.0}{\makebox(0,0)[lb]{\smash{{\SetFigFont{6}{7.2}{rm}Limiter}}}}}
\put(11251,-2836){\rotatebox{90.0}{\makebox(0,0)[lb]{\smash{{\SetFigFont{6}{7.2}{rm}Limiter}}}}}
\put(10651,-7111){\makebox(0,0)[lb]{\smash{{\SetFigFont{7}{8.4}{rm}s=1}}}}
\put(1951,-2236){\makebox(0,0)[lb]{\smash{{\SetFigFont{7}{8.4}{rm}r=1}}}}
\put(2476,-6436){\makebox(0,0)[lb]{\smash{{\SetFigFont{7}{8.4}{rm}r=0}}}}
\put(2776,-6886){\makebox(0,0)[lb]{\smash{{\SetFigFont{7}{8.4}{rm}s=0}}}}
\put(1951,-4636){\makebox(0,0)[lb]{\smash{{\SetFigFont{7}{8.4}{rm}r=1/2}}}}
\put(9451,-4486){\makebox(0,0)[lb]{\smash{{\SetFigFont{7}{8.4}{rm}Separatrix}}}}
\put(3076,-3286){\makebox(0,0)[lb]{\smash{{\SetFigFont{6}{7.2}{rm}$\partial_s T =  \alpha T ^{-3/2} $}}}}
\put(6151,-6136){\makebox(0,0)[lb]{\smash{{\SetFigFont{6}{7.2}{rm}$\partial_r T = -  Q_{\perp}$}}}}
\put(8851,-3286){\makebox(0,0)[lb]{\smash{{\SetFigFont{6}{7.2}{rm}$\partial_s T =  - \alpha T^{-3/2} $}}}}
\end{picture}%